\documentclass[10pt]{article}
\usepackage[utf8]{inputenc}
\usepackage{fancyhdr}
\usepackage{amsmath,amssymb,amsthm}
\newtheorem{theorem}{Theorem}[section]
\newtheorem{corollary}{Corollary}[theorem]

\newtheorem{proposition}{Proposition}[section]
\theoremstyle{definition}
\newtheorem{remark}{Remark}[section]
\usepackage{geometry}
\allowdisplaybreaks
\usepackage{graphicx}
\usepackage{subfig}
\usepackage{float}
\usepackage[]{algorithm2e}
\usepackage{listings}
\usepackage[toc,page]{appendix}
\usepackage{bbm}
\newcommand{\R}{\mathbb{R}}
\newcommand{\E}{\mathbb{E}}
\newcommand{\e}{\mathbf{e}}
\newcommand{\N}{\mathcal{N}}

\newcommand{\PS}{\mathcal{P}}
\newcommand{\ZS}{\mathcal{Z}}

\usepackage[table,xcdraw]{xcolor}
\usepackage{hyperref}

\title{Clustering, factor discovery and optimal transport}
\author{Hongkang Yang\thanks{Courant Institute of Mathematical Sciences, 251 Mercer Street, New York, NY 10012, USA, \texttt{hy1194@nyu.edu}}
\and Esteban G. Tabak\thanks{Courant Institute of Mathematical Sciences, 251 Mercer Street, New York, NY 10012, USA, \texttt{tabak@cims.nyu.edu}}}
\date{September 9, 2020}
\begin{document}
\maketitle

\begin{abstract}
The clustering problem, and more generally, latent factor discovery --or latent space inference-- is formulated in terms of the Wasserstein barycenter problem from optimal transport. The objective proposed is the maximization of the variability attributable to class, further characterized as the minimization of the variance of the Wasserstein barycenter. Existing theory, which constrains the transport maps to rigid translations, is extended to affine transformations. The resulting non-parametric clustering algorithms include $k$-means as a special case and exhibit more robust performance. A continuous version of these algorithms discovers continuous latent variables and generalizes principal curves. The strength of these algorithms is demonstrated by tests on both artificial and real-world data sets.
\end{abstract}


\medskip


\section{Introduction}

In order to analyze a given data set, it is often instructive to consider a factor model:
\begin{equation}
\label{factor model}
X = G(Z_1, \dots Z_n),
\end{equation}
where the data is modelled as a random variable $X$ generated from unknown factors, the random variables $Z_i$, through an unknown generating process, the function $G$. The given data is thought of as a sample set drawn from a probability measure $\rho$, and our task is to construct a factor model (\ref{factor model}) such that the law of $X$ equals $\rho$.

This task can be addressed in at least two different ways, leading to two distinct problems:
\begin{itemize}
\item Generative modeling and density estimation: The factors $Z_i$ are fixed, for instance as Gaussians, and the emphasis is on constructing the generator $G$, which helps us generate more samples from $\rho$ or estimate its density function $\rho(x)$. This approach is exemplified by variational autoencoders \cite{kingma2013VAE}, generative adversarial networks \cite{goodfellow2014GAN} and normalizing flows \cite{tabak2010density}.

\item Factor discovery: The emphasis is instead on identifying reasonable and interpretable factors $Z$. An informal criterion for choosing $Z$ is that it can satisfactorily explain the data $\rho$. This approach is exemplified by clustering, principal components and principal surfaces \cite{hastie2005elements,alpaydin2009introduction}.
\end{itemize}
This paper concerns the second problem. Most of the existing works focus on either of two directions:
\begin{enumerate}
\item Effective objective functions that characterize a satisfactory factor $Z$.
\item Effective modelling assumptions on $Z$ and generating process $G$.
\end{enumerate}
To make the discussion more concrete, let us consider an elementary setting:
\begin{equation*}
X = G(Z) + Y
\end{equation*}
where $Z$ is the factor we seek and $Y$ is a random variable that takes care of the component of $X$ unexplained by $Z$. Define the objective function as the mean-square error
\begin{equation}
\label{MSE problem}
L(Z,G) = \inf_{X \sim \rho} \E \big[ \|X-G(Z)\|^2 \big].
\end{equation}
Different modeling assumptions on $Z,G$ could lead to diverse solutions. Suppose that $Z$ is a random variable on some latent space $\ZS$ with law $\nu$:
\begin{itemize}
\item Assumption 1: The latent space is finite, $\ZS=\{1,\dots K\}$. Then (\ref{MSE problem}) assigns $X$ to the nearest point $G(k)$ and thus the problem reduces to $k$-means clustering. The probability measure $\nu$ becomes the weights of the clusters $\rho_k$, and the unexplained $Y$ conditioned on $Z=k$ is the $k$-th cluster with its mean removed.

\item Assumption 2: The latent space is low-dimensional Euclidean $\ZS=\R^k$ and the generator $G$ is linear. Then, the problem reduces to finding a $k$-dimensional principal hyperplane, $\nu$ is the projection of the data $\rho$ onto the hyperplane, and $Y|Z$ are the residuals of $X$ perpendicular to the plane.

\item Assumption 3: $\ZS=\R^k$ and $G$ is smooth. Then, the problem reduces to principal curve/surface/hypersurface, and $Y|Z$ are the residuals of $X$ perpendicular to the surface.
\end{itemize}
Hence, even with the simplest mean-square objective (\ref{MSE problem}), the factor discovery problem contains rich algorithms and applications. The aim of this paper is to design an objective function that is more suitable than (\ref{MSE problem}) for the task of factor discovery, and explore what clustering and principal surface algorithms can be derived from it.\\

The factor discovery problem is not an end in itself. Ultimately, we aim at fully understanding the data $\rho$, identifying in particular all hidden factors that participate in the data-generating process (\ref{factor model}). Once a factor $Z_{1}$ is found by a factor discovery algorithm, we may filter out the influence of $Z_1$ in the data $X$ and proceed to look for other, less prominent factors $Z_{2},Z_3,\dots$ in the filtered data.

As a concrete example of filtering, consider the task of removing batch effects \cite{chen2011removing,tabak2018explanation}, common in biostatistics. Biomedical datasets such as gene expression microarrays are highly vulnerable to batch effects: the data are usually processed and collected from multiple small batches, inevitably introducing a large amount of batch-related variance into the dataset, which corrupts the data and prevents us from discovering meaningful factors other than the batches.
One solution is to study separately each conditional distribution $\rho(\cdot |z)$ (the law of $X|Z=z$, $z\in\ZS$). A problem with this is that each conditional distributions could correspond to very few data points and, in examples other than batch removal, when $\ZS$ can be continuous, no data at all in our finite dataset. In addition, any additionaL factors discovered could be inconsistent across different $z \in \ZS$. Instead, The procedure that we propose consolidates all conditional distributions $\rho(\cdot |z)$ into one representative distribution, their Wasserstein or optimal transport barycenter $\mu$. This will be described in section \ref{sec: background}, but see Figure \ref{fig: demo barycenter} for a demonstration.

The intuition is that the barycenter is the $z$-independent distribution that is closest to all conditional distributions $\rho(\cdot |z)$, so mapping each $\rho(\cdot |z)$ to the berycenter is a way to consolidate all batches into one dataset with minimum distortion. Thus, the barycenter should preserve all meaningful patterns in the data $X$ that are not explainable by $Z$, allowing us to perform factor discovery on this filtered data to uncover additional hidden factors.

\begin{figure}
\centering
\subfloat{\includegraphics[scale=0.3]{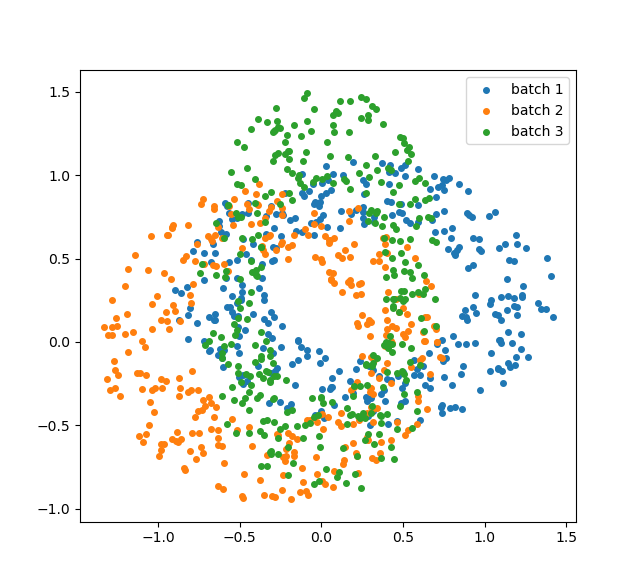}}
\subfloat{\includegraphics[scale=0.32]{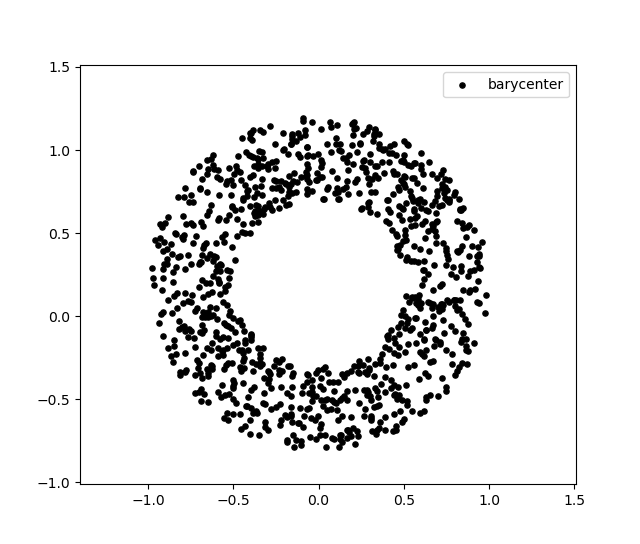}}
\caption{An example of Wasserstein barycenter. Left: Three distributions with annulus shape (represented by their samples). Right: their barycenter, computed with the optimal affine transport maps from Section \ref{sec: background}.}
\label{fig: demo barycenter}
\end{figure}

To see how this procedure affects the choice of objective function, recall the mean-square objective (\ref{MSE problem}):
\begin{equation*}
\min_G L(Z,G) = \min_G \inf_{X \sim \rho} \E\Big[ \E \big[ \|X-G(Z)\|^2 \big| Z=z \big] \Big] = \inf_{X \sim \rho} \E\Big[ Var(X| Z=z) \big] \Big]
\end{equation*}
where the optimal $G(z)$ is the mean of $X|Z=z$. Then, the remaining problem, i.e. factor discovery, is equivalent to minimizing the weighted average of the variances of all conditionals $\rho(\cdot |z)$:
\begin{equation}
\label{objective: expectation variance}
\min_{Z} \inf_{X \sim \rho} \E\Big[ Var(X| Z=z) \big] \Big] = \min_{\rho(\cdot|z),\nu} \Big\{ \E\big[Var(X|z)\big] ~\big|~ X|z \sim \rho(\cdot |z), ~\rho = \int_{\ZS} \rho(\cdot|z) d\nu(z) \Big\}.
\end{equation}
For factor discovery, we replace the conditionals $\rho(\cdot|z)$ by their barycenter $\mu$, which is more suitable for conducting downstream analysis. So it is natural to define a new objective function by the variance of this barycenter:
\begin{equation*}
\min_{\rho(\cdot|z),\nu} \Big\{ Var(Y) ~\big|~ Y \sim \mu, ~\mu \text{ is the barycenter of } \rho(\cdot|z)\nu(z), ~\rho = \int \rho(\cdot|z) d\nu(z) \Big\}.
\end{equation*}
Thus, a minimizer produces the maximum reduction in variance when we go from the raw data $\rho$ to the processed data (the barycenter) $\mu$, which intuitively corresponds to the maximum gain in information.
Finally, replace $\rho(\cdot|z)$ and $\nu$ by the assignment distribution $\nu(\cdot|x)$, which softly assigns latent values $z$ to the sample $x$:
\begin{equation}
\label{objective: factor discovery with barycenter}
\min_{\nu(\cdot|x)} \Big\{ Var(Y) ~\big|~ Y \sim \mu, ~\mu \text{ is the barycenter of } \rho(\cdot|z)\nu(z), ~\rho(\cdot|z) \nu = \nu(\cdot|x)\rho \Big\}
\end{equation}

The application of Wasserstein barycenter to the factor discovery problem was first proposed in \cite{tabak2018explanation}. The main difficulty in solving (\ref{objective: factor discovery with barycenter}) is to compute the barycenter's variance (given $\rho(\cdot|z)$ and $\nu(z)$), for which there are currently two approaches:
\begin{enumerate}
\item General distributions: Compute the barycenter for general conditionals $\rho(\cdot|z)$, which involves finding the optimal transport maps, and optimize the assignment distribution $\nu(\cdot|x)$ on top of this computation. This is the approach in \cite{yang2019BaryNet}, which models and optimizes these functions through neural networks.

\item Special distributions: Assume that the conditionals $\rho(\cdot|z)$ belong to certain distribution families that admit a tractable solution for barycenters. The setting considered in \cite{tabak2018explanation} is that all $\rho(\cdot|z)$ are equivalent up to rigid translations, and thus the barycenter $\mu$ corrresponds to a $z$-dependent rigid translation of each $\rho(\cdot|z)$. It follows that the objective (\ref{objective: factor discovery with barycenter}) reduces to the mean-square objective (\ref{objective: expectation variance}), and from it we can recover classical algorithms such as $k$-means and principal components.
\end{enumerate}

This article follows the second approach and generalizes the assumption of translation equivalence in \cite{tabak2018explanation} to the less stringent condition that the conditionals $\rho(\cdot|z)$ are equivalent up to affine transforms. By taking into consideration the second moments, this generalized objective function leads to clustering and principal surface algorithms that are more robust to the clusters' diversity. In particular, we introduce two practical algorithms: barycentric $k$-means (Algorithm \ref{alg: barycentric clustering isotropic hard}) and affine factor discovery (Algorithm \ref{alg: continuous barycentric clustering}). 
Barycentric $k$-means is a clustering algorithm almost identical to $k$-means, except for a simple change in the update rule: $k$-means assigns each sample $x_i$ to cluster $k_i$ by
\begin{equation*}
k_i \gets \text{argmin}_k ||x_i-\overline{x}_k||^2
\end{equation*}
while barycentric $k$-means assigns $x_i$ by
\begin{equation*}
k_i \gets \text{argmin}_{k}\left(\frac{||x_i - \overline{x}_k||^2}{\sigma_k} + \sigma_k \right)
\end{equation*}
where $\overline{x}_k,\sigma_k$ are the means and standard deviations of cluster $k$. Affine factor discovery is a principal curve/surface algorithm analogous to barycentric $k$-means.

Existing literature includes several clustering algorithms that utilize the second moment information, which we will introduce and compare with our algorithms in Section \ref{sec: related work}. The purpose of this article is not so much to build one new method for clustering or principal surfaces, but rather to develop a paradigm that conceptualizes these algorithms as a procedure for the general factor discovery problem (\ref{factor model}) and to show that optimal transport and Wasserstein barycenter are useful tools for uncovering and assessing hidden factors.\\

The plan of this article is as follows. After this introduction, Section \ref{sec: background} gives a quick overview of optimal transport, Wasserstein barycenter and its affine formulation. Section \ref{sec: clustering} optimizes our factor discovery objective function (\ref{objective: factor discovery with barycenter}) with affine maps and derives clustering algorithms based on $k$-means. Section \ref{sec: continuous clustering} considers the continuous setting when $\ZS=\R$ and derives a more general principal surface algorithm. Section \ref{sec: performance} experiments with these algorithms on synthetic and real-world data. Section \ref{sec: conclusion} summarizes the work and sketches some directions of current research.

\section{Background on the Wasserstein barycenter}
\label{sec: background}

All data distributions in this paper live in $\R^d$, while the factors $z$ belong to a measurable space $\ZS$ to be specified. Denote the raw data by $\rho$. For each data point $x$, the assignment distribution $\nu(z|x)$ indicates the likelihood that factor $z$ corresponds to $x$. One can define the joint distribution $\pi$
\begin{equation}
\label{joint distribution}
\pi(x,z) = \rho(x) \nu(z|x) = \rho(x|z) \nu(z),
\end{equation}
where $\nu$ is the latent distribution of the factors and the conditional distributions $\rho(\cdot|z)$ are commonly referred to as clusters. We say that a joint distribution $\pi=\rho(\cdot|z) \nu$ is a disintegration or clustering plan of $\rho$ if
\begin{equation*}
\rho = \int_{\ZS} \rho(\cdot|z) d\nu(z) = \int_{\ZS} \pi\ dz .
\end{equation*}

Denote by $\PS(\R^d),\PS_2(\R^d),\PS_{ac}(\R^d)$ the space of Borel probability measures in $\R^d$, those with finite second moments and those that are absolutely continuous respectively.
A random variable $X$ with distribution law $\rho$ is denoted by $X \sim \rho$. The notation $T\#\rho = \mu$ indicates that a measurable map $T$ transports (or pushes-forward) a probability measure $\rho$ to another measure $\mu$: If $X\sim \rho$, then $\mu = law(T(x))$. Equivalently, for all measurable subset $A \subseteq \mathbb{R}^d$,
$$\mu(A) = \rho(T^{-1}(A)).$$

The Wasserstein metric $W_2$ or optimal transport cost with square distance \cite{villani2003topics} between probability measures $\rho$ and $\mu$ is given by
\begin{equation}
\label{def: W2 metric}
W_2^2(\rho,\mu) = \inf_{X\sim\rho,Y\sim\mu} \E\big[\|X-Y\|^2\big].
\end{equation}
By Brennier's theorem \cite{brenier1991polar,villani2003topics}, whenever $\rho \in \PS_{2,ac}(\R^d)$, there exists a unique map $T$, the optimal transport map, such that $T\#\rho=\mu$ and
\begin{equation*}
W_2(\rho,\mu) = \E_{X\sim\rho}\big[\|X-T(X)\|^2\big].
\end{equation*}

Given a disintegration (or clustering plan) $\pi = \rho(\cdot|z)\nu$, the Wasserstein barycenter $\mu$ is a minimizer of the average transport cost
\begin{equation}
\label{def: barycenter}
\min_{\mu\in\mathcal{P}_2(\mathbb{R}^d)} \int W_2^2\big(\rho(\cdot|z),\mu\big) ~d\nu(z)
\end{equation}
The barycenter always exists by Theorem 1 of \cite{yang2019BaryNet}. The Wasserstein metric (\ref{def: W2 metric}) can be considered as a measurement of pointwise distortion of the data $\rho$, and thus the barycenter is a way to summarize all conditionals $\rho(\cdot|z)$ into one distribution while minimizing their distortion. Hence, one should expect that the barycenter can reliably preserve the underlying patterns common to all $\rho(\cdot|z)$.

In order to work with the objective function (\ref{objective: factor discovery with barycenter}), we need to make tractable the computation and optimization of the barycenter's variance $Var(Y),Y\sim\mu$. As linear problems are often amenable to closed-form solutions and efficient computations, the Wasserstein barycenter would become much simpler if all $\rho(\cdot|z)$ and $\mu$ are equivalent up to affine transformations, that is, $T\#\rho(\cdot|z_1) = \rho(\cdot|z_2)$ for an affine map $T$.

Intuitively, a sufficient condition for this property is that the conditionals $\rho(\cdot|z)$ have similar shapes (e.g. they are all Gaussians), which implies that the barycenter $\mu$ as their representative should also have that shape. Consequently, the map $T_k$ only needs to translate and dilate the $\rho_k$ to transform them into $\mu$. This is confirmed by Theorem 5 of \cite{yang2019BaryNet}:

\begin{theorem}
\label{thm: covariance formula}
Given any measureable space $Z$ and a joint distribution (\ref{joint distribution}), assume that $\rho$ has finite second moment and each conditional $\rho(\cdot|z)$ is a Gaussian distribution
\begin{equation*}
\rho(\cdot|z) = \N\big(\overline{x}(z),\Sigma(z)\big)
\end{equation*}
with mean $m(z)$ and covariance matrix $\Sigma(z)$.
Then, there exists a barycenter $\mu$, which is a Gaussian with mean and covariance
\begin{align}
\label{barycenter mean integral}
\overline{y} &= \int \overline{x}(z) d\nu(z) = \mathbb{E}_{X\sim\rho}[X],\\
\label{barycenter covariance integral}
\Sigma_y &= \int \big(\Sigma_y^{\frac{1}{2}} \cdot \Sigma(z) \cdot \Sigma_y^{\frac{1}{2}} \big)^{\frac{1}{2}} ~d\nu(z),
\end{align}
where $\Sigma^{\frac{1}{2}}$ is the principal matrix square root.
\end{theorem}

\begin{corollary}
\label{cor: isotropic std}
If we further assume that each $\rho(\cdot|z)$ is isotropic: $\Sigma(z) = \sigma^2(z) \cdot I_d$, then the unique barycenter is also isotropic, with standard deviation
\begin{equation}
\label{average std}
\sigma = \int \sigma(z) d\nu(z).
\end{equation}
\end{corollary}

\begin{remark}
If we are only concerned with the setting of clustering with a finite latent space $Z=\{1,\dots K\}$, then the above result has been given by Corollary 4.5 of \cite{alvarez2016fixed}. Nevertheless, Section \ref{sec: continuous clustering} studies the continuous case with $Z=\mathbb{R}$, and we need the full strength of Theorem \ref{thm: covariance formula}.
\end{remark}

\medskip
Moreover, the optimal transport maps $T$ between these distributions are also affine. By Theorem 2.1 of \cite{cuesta1996bounds}, the optimal transport map $T_z$ that transports $\rho(\cdot|z)$ to $\mu$ is given by
\begin{align}
\label{optimal affine transport maps}
\begin{split}
T_z(x) &= A_z x + b_z\\
A_z &= \Sigma(z)^{-\frac{1}{2}}\big(\Sigma(z)^{\frac{1}{2}}\Sigma_y\Sigma(z)^{\frac{1}{2}}\big)^{\frac{1}{2}}\Sigma(z)^{-\frac{1}{2}}, ~
b_z = \overline{y} - A_z \overline{x}(z)
\end{split}
\end{align}

From now on, we will always model the conditionals $\rho(\cdot|z)$ as Gaussian distributions. In fact, all our arguments and algorithms apply to the slightly more general case of location-scale families \cite{alvarez2016fixed}, which by definition consist of distributions that are equivalently up to affine transformations.
The restriction to Gaussians conditionals (or location-scale families) is not a stringent requirement in practice. Common clustering methods such as EM algorithm \cite{redner1984mixture,alpaydin2009introduction} often model the clusters as Gaussians, while the ``standard data" for $k$-means consists of spherical clusters with equal radii \cite{selim1991simulated}.

\section{Clustering with affine barycenter}
\label{sec: clustering}

This section studies the setting of discrete factors $\ZS=\{1,\dots K\}$ and derives clustering algorithms based on the objective function (\ref{objective: factor discovery with barycenter}) and the assumption of Gaussian conditionals $\rho(\cdot|z)$.

In practice, we are given the following empirical distribution
\begin{equation*}
\rho^{(n)} = \frac{1}{N} \sum_{i=1}^N \delta_{x_i},
\end{equation*}
where $\{x_i\}_{i=1}^N$ are i.i.d. samples from $\rho$. Their weights are uniformly set to $1/N$ in this paper, but one could also consider a more general case with more/less weights attached, for instance, to outliers or to measurements performed under more or less controlled circumstances.

Any joint distribution $\pi$ from (\ref{joint distribution}) can be expressed as an $N\times K$ matrix. For convenience, denote the conditionals $\rho(\cdot | k)$ by $\rho_k$ for $k \in \{1,\dots K\}$, represent the latent factor distribution $\nu$ by the weights $P_k$, and represent the assignment distribution $\nu(\cdot|x)$ by the following matrix:
\begin{equation*}
P = [P_k^i], ~P_k^i = \nu(k|x_i) = N \cdot \pi(x_i,k).
\end{equation*}
The probability vector $P^i$ representing $\nu(\cdot|x_i)$ is a soft class assignment for $x_i$, which becomes hard when $P^i$ is given by a basis vector $\e_{k_i}$. Each $P^i$ ranges in the $K$-dimensional simplex $\Delta^K$ and the matrix $P$ ranges in the compact convex set $\prod_{i=1}^N \Delta^K$. The domain of hard assignments is the set of extremal points of $\prod_i \Delta^K$ \cite{bezdek2013pattern}.

Once a class assignment $P=N \pi$ is given, the clusters $\rho_k(x)$ can be estimated by formula (\ref{joint distribution}),
\begin{align}
\label{compute clusters}
\rho_k(x_i) &= \frac{\nu(k|x_i)\rho(x_i)}{\nu(k)} = \frac{P_k^i N^{-1}}{P_k} = \frac{P_k^i}{\sum_{j=1}^N P_k^j}.
\end{align}
Then, the cluster means $\overline{x}_k$ and covariances $\Sigma_k$ can be estimated by
\begin{equation}
\label{cluster mean and covariance}
\overline{x}_k = \frac{\sum_i P^i_k x_i }{\sum_i P^i_k}, ~\Sigma_k = \frac{\sum_i P^i_k (x_i - \overline{x}_k) \cdot (x_i - \overline{x}_k)^T }{ \sum_i P^i_k }.
\end{equation}
By Theorem \ref{thm: covariance formula}, the covariance of the barycenter $\mu$ is given by the discrete version of (\ref{barycenter covariance integral}):
\begin{equation}
\label{barycenter covariance discrete}
\Sigma_y = \sum\limits_{k=1}^K P_k \big(\Sigma_y^{\frac{1}{2}}\Sigma_k\Sigma_y^{\frac{1}{2}}\big)^{\frac{1}{2}}.
\end{equation}
This is a non-linear matrix equation that admits a unique positive-definite solution $\Sigma_y$ if at least one of $P_k \Sigma_k$ is positive-definite \cite{alvarez2016fixed}. Then, $\Sigma_y$ can be calculated through the following iteration scheme:
\begin{equation}
\label{iteration scheme}
    \Sigma(n+1) \gets \Sigma(n)^{-\frac{1}{2}} \Big(\sum_{k=1}^K P_k\Sigma(n)^{\frac{1}{2}}\Sigma_k\Sigma(n)^{\frac{1}{2}}\Big)^2 \Sigma(n)^{-\frac{1}{2}},
\end{equation}
where the initialization $\Sigma(0)$ is an arbitrary positive-definite matrix. This iteration is guaranteed to converge to the correct $\Sigma_y$ (when one of $P_k \Sigma_k$ is positive-definite) \cite{alvarez2016fixed}.

Hence, the factor discovery objective (\ref{objective: factor discovery with barycenter}) can be specialized into
\begin{equation}
\label{objective: clustering}
\min_P Tr[\Sigma_y],
\end{equation}
with $\Sigma_y$ defined by (\ref{cluster mean and covariance}) and (\ref{barycenter covariance discrete}).

\begin{remark}
\label{remark: location family}
To gain insight into the functioning of (\ref{objective: clustering}), we can analyze its behavior in the much simplified setting, when the Gaussian clusters $\rho_k$ differ only in their means $\overline{x}_k$. Then Theorem 2.1 of \cite{cuesta1996bounds} implies that the optimal transport maps $T_k$ are translations by $\overline{y}-\overline{x}_k$. In that case, \cite{tabak2018explanation} showed that the barycenter's variance is reduced to the sum of within-cluster variances (equivalently, sum of squared errors, SSE)
\begin{equation}
\label{SSE}
Tr[\Sigma_y] = \sum_{k=1}^K \sum_{i=1}^N P^i_k ||x_i - \overline{x}_k||^2
\end{equation}
which is exactly the objective function of $k$-means. Hence, by upgrading to general Gaussian clusters, one develops a generalization of $k$-means that takes advantage of second moments.
\end{remark}

\subsection{Gradient descent solution}
\label{sec: gradient descent solution}

Our approach is to perform gradient descent on $P\in \prod_i \Delta^K$ to solve for the minimization (\ref{objective: clustering}). Even though the implicit nonlinear matrix equation (\ref{barycenter covariance discrete}) determines $\Sigma_y$ uniquely, it is not clear a priori whether the solution $\Sigma_y$ is differentiable. Thus, we prove in Appendix \ref{appendix: existence of derivative} that the partial derivatives $\partial \Sigma_y / \partial P_k^i$ always exist. Then, in Appendix \ref{appendix: computation of derivative}, we derive explicit formulae for these derivatives.

The gradient of $Tr[\Sigma_y]$ with respect to each sample point $x_i$'s probability vector $P^i$ is given by
\begin{equation}
\label{gradient of objective}
\nabla_{P^i} Tr[\Sigma_y] = \sum_{k=1}^K vec(I)^T \cdot W_k \cdot vec\big[(x_i-\overline{x_k})\cdot (x_i-\overline{x_k})^T + \Sigma_k\big]\e_k
\end{equation}
where $\e_k$ is the basis vector, $vec$ is vectorization (which stacks the columns of the input matrix into one long vector), and the $W_k$ are weight matrices
\begin{align}
\label{weight matrix}
\begin{split}
W_k &:= \big(\Sigma_y^{\frac{1}{2}} \otimes \Sigma_y^{\frac{1}{2}}\big) \Big[\sum\limits_{h=1}^K P_h (U_h \otimes U_h) (D_h^{\frac{1}{2}}\otimes I + I \otimes D_h^{\frac{1}{2}})^{-1}(D_h^{\frac{1}{2}}\otimes D_h^{\frac{1}{2}}) (U_h^T \otimes U_h^T)\Big]^{-1} \\
&\quad \Big[(U_k \otimes U_k) (D_k^{\frac{1}{2}}\otimes I + I \otimes D_k^{\frac{1}{2}})^{-1} (U_k^T \otimes U_k^T)\Big](\Sigma_y^{\frac{1}{2}} \otimes \Sigma_y^{\frac{1}{2}}).
\end{split}
\end{align}
Here $\otimes$ is the Kronecker product, and the $U$ are the orthonormal and $D$ the diagonal matrices in the eigendecompositions
\begin{equation*}
\Sigma_y^{\frac{1}{2}}\Sigma_k\Sigma_y^{\frac{1}{2}} = U_kD_kU_k^T \text{ and } \Sigma_y = U_yD_yU_y^T.
\end{equation*}


The update rule at each time of a gradient descent step $t$ is a projected gradient descent:
\begin{equation*}
P(t+1) = Proj_{\prod_i\Delta^K}\Big({P(t)} -\eta \cdot\nabla_{P}Tr[\Sigma_y]\Big)
\end{equation*}
where $\eta$ is the learning rate and $Proj$ is the projection onto the closest stochastic matrix in $\prod_i\Delta^K$, which can be computed efficiently as described in \cite{wang2013projection}.
The step size $\eta$ can be either fixed at a small value, or determined at each step via backtracking line search \cite{boyd2004convex}, using a threshold $\alpha\in (0,1/2)$ and a shortening rate $\beta\in(0,1)$, and reducing $\eta$ into $\beta \eta$ if the amount of descent is not enough:
$$Tr[\Sigma_y]\big(P(t+1)\big) - Tr[\Sigma_y]\big(P(t)\big) > \alpha ~vec(\nabla_{P}Tr[\Sigma_y])^T \cdot vec\big[P(t+1)-P(t)\big].$$
This descent-based clustering algorithm is summarized below, with initialization based on that of $k$-means.

\medskip
\begin{algorithm}[H]
\KwData{Sample $\{x_i\}$ and number of classes $K$}
Initialize the means $\{\overline{x}_k\}$ randomly\\
Initialize assignment matrix $P$ either randomly or set each $P^i$ to be the one-hot vector corresponding to the mean $\overline{x}_k$ closest to $x^i$\\
\While{not converging}{
	Compute the barycenter's covariance $\Sigma_y$ by iteration (\ref{iteration scheme})\\
	Compute weight matrices $W_1,\dots W_K$ by (\ref{weight matrix})\\
   	Compute the gradient $\nabla_{P}Tr[\Sigma_y] = (\nabla_{P^i}Tr[\Sigma_y])_i = \sum_{i,k} vec(I)^T W_k vec\big[(x_i-\overline{x_k})\cdot (x_i-\overline{x_k})^T + \Sigma_k\big]\e_{ik}$\\
    (Optimize step size $\eta$ by backtracking)\\
    Update $P \gets Proj_{\prod_i\Delta^K}\big({P} -\eta \cdot\nabla_{P}Tr[\Sigma_y]\big)$\\
	Update cluster means $\{\overline{x}_k\}$ and covariances $\{\Sigma_k$\} by (\ref{cluster mean and covariance})\\
}
\Return{\text{\upshape Assignment} $P$}
\caption{Barycentric clustering}
\label{alg: barycentric clustering soft}
\end{algorithm}

\begin{remark}
It might appear at first sight that Algorithm \ref{alg: barycentric clustering soft} performs an alternating descent, similarly to $k$-means, alternating between optimizing the assignment $P$ and updating the means $\overline{x}_k$. Nevertheless, the derivation of (\ref{gradient of objective}) in Appendix \ref{appendix: computation of derivative} does not treat $\overline{x}_k$ as constants. Instead, it directly solves for the gradient $\nabla_P Tr[\Sigma_y]$, incorporating the derivatives $\nabla_P \overline{x}_k$. Hence, Algorithm \ref{alg: barycentric clustering soft} is simply a gradient descent on the objective (\ref{gradient of objective}), which with sufficiently small step size $\eta$ necessarily converges to a critical point.
\end{remark}

\medskip
Recall that $k$-means minimizes the sum of squared errors (\ref{SSE}), whose partial derivatives are simply
$$\partial_{P^i_k} SSE = ||x_i-\overline{x}_k||^2 + 2\sum_j P^j_k (x_j-\overline{x}_k)\frac{\partial \overline{x}_k}{\partial P^i_k} = ||x_i-\overline{x}_k||^2$$
Instead of performing gradient descent, $k$-means directly assigns $x_i$ to the closest cluster $k_i$, that is,
$$k_i= \text{argmin}_k ||x_i-\overline{x}_k||^2 = \text{argmin}_k \partial_{P^i_k} SSE$$
We can interpret this hard assignment as equivalent to a gradient descent on $P^i$ with arbitrarily large step size, followed by the projection $Proj_{\Delta^K}$, so that $P^i$ arrives at an extremal point of $\Delta^K$, which is a one-hot vector.

In exactly the same way, we can simplify Algorithm \ref{alg: barycentric clustering soft} into a hard assignment algorithm, such that each $x_i$ is assigned to the cluster $k_i$ with the smallest gradient term in $\partial_{P^i}Tr[\Sigma_y]$.

\medskip
\begin{algorithm}[H]
\KwData{Sample $\{x_i\}$ and number of classes $K$}
Initialize the means $\{\overline{x}_k\}$ randomly and the labels $k_i$ by the closest mean\\
\While{not converging}{
	Compute the gradient $\nabla_{P}Tr[\Sigma_y]$\\
  \For{$x_i$ in sample}{
   	$k_i \gets \text{argmin}_{k}\big( \partial_{P^i_k}Tr[\Sigma_y] \big)$\\
   }
	Update cluster means $\{\overline{x}_k\}$ and covariances $\{\Sigma_k$\}\\
	(Possibly apply an update rate $c$ to smooth the update: $\overline{x}_k\gets c$ new $\overline{x}_k + (1-c)$ old $\overline{x}_k$)
}
\Return{\text{\upshape Labels} $\{k_i\}$}
\caption{Hard barycentric clustering}
\label{alg: barycentric clustering hard}
\end{algorithm}

One common issue for clustering algorithms that utilize the second moments is that numerical errors could arise when some cluster's covariance becomes non-invertible, or equivalently, the cluster lies on a low-dimensional plane. Specifically such problem could arise for the $D_k$ terms in the weight matrices (\ref{weight matrix}), though it has not been noticed in our experiments in Section \ref{sec: performance}. In practice, one can easily avoid this problem by making the covariances $\Sigma_k$ strictly positive-definite: e.g.
\begin{equation*}
\Sigma_k \gets \Sigma_k + \epsilon I_d
\end{equation*}
for some $\epsilon \ll 1$.

\subsection{Relation to \textit{k}-means}
\label{sec: relation to k-means}

We show next that the barycentric clustering algorithms reduce to $k$-means in the latter's setting. The ``standard data'' for $k$-mean consist of spherical clusters with identical radii and proportions \cite{selim1991simulated}, which implies that $P_1 =\dots =P_K = 1/K$ and $\Sigma_1=\dots=\Sigma_K = \frac{\sigma^2}{d} I$ for some common variance $\sigma^2$. Then the gradient (\ref{gradient of objective}) simplifies into
$$\partial_{P_k^i} Tr[\Sigma_y] = \frac{\sigma^2}{d} Tr\big[(x_i-\overline{x_k})\cdot (x_i-\overline{x_k})^T + \Sigma_k \big] = \frac{\sigma^2}{d} ( ||x_i - \overline{x}_k||^2 + \sigma^2).$$

Since each $P^i$ lies in the simplex $\Delta^K$, the direction of gradient descent must be parallel to $\Delta^K$. Hence the term $\sigma^2$, shared by all entries of $\nabla_{P^i}Tr[\Sigma_y]$, is eliminated by the projection map $Proj_{\prod_i\Delta^K}$ of Algorithm \ref{alg: barycentric clustering soft}, while for Algorithm \ref{alg: barycentric clustering hard}, it is eliminated by the $\text{argmin}_{k}$ step. The resulting gradient $$\nabla_{P^i}Tr[\Sigma_y] = \sum_{k=1}^K ||x_i - \overline{x}_k||^2 \e_k$$ is precisely the gradient of the sum of squared errors (\ref{SSE}), the objective function of $k$-means. It is straightforward to check that Algorithm \ref{alg: barycentric clustering hard} reduces to $k$-means, and thus $k$-means can be seen as a special case of barycentric clustering.\\


\subsection{Isotropic solution}
\label{sec: isotropic solution}

Section \ref{sec: performance} will demonstrate that the barycentric clustering algorithms can recognize clusters that deviate from the ``standard data'', for which $k$-means and fuzzy $k$-means would fail, but this robustness comes at the expense of the complexity of  gradients in (\ref{gradient of objective}) and (\ref{weight matrix}).
Here we explore a situation in between, making hypotheses weaker than the ``standard data'', yet strong enough to yield solutions that, while more robust than $k$-means, are at the same time simpler than (\ref{gradient of objective}) and easier to interpret.

Since the complexity of (\ref{gradient of objective}) results mostly from the non-commutativity of the matrix product, we can impose the assumption that all covariances are of the form
$$\Sigma_k = \frac{\sigma^2_k}{d} I$$
where $\sigma^2_k$ is the variance of cluster $\rho_k$. 
This assumption can be seen as a generalization of the ``standard data'''s requirement that all clusters be radial with equal variances.

From Corollary \ref{cor: isotropic std}, the barycenter's covariance becomes
\begin{equation}
\label{std formula discrete}
\Sigma_y = \frac{\sigma^2_y}{d} I, ~\sigma_y = \sum_{k=1}^K P_k \sigma_k
\end{equation}
and the gradient (\ref{gradient of objective}) reduces to
\begin{equation*}
\partial_{P_k^i} Tr[\Sigma_y] = \frac{\sigma_y^3}{d} \Big(\frac{||x_i-\overline{x}_k||^2}{\sigma_k}+\sigma_k\Big).
\end{equation*}
Since the algorithms are only concerned with the gradient's direction, the gradient is effectively
\begin{equation}
\label{std variance form gradient}
\nabla_{P^i} Tr[\Sigma_y] = \sum_k \Big(\frac{||x_i-\overline{x}_k||^2}{\sigma_k}+\sigma_k\Big) \e_k.
\end{equation}

\begin{remark}
Alternatively, we can obtain the gradient (\ref{std variance form gradient}) directly differentiating the weighted sum of standard deviations (\ref{std formula discrete}). Note that the standard deviation can also be calculated via
$$\sigma_k = \sqrt{Var(\rho_k)} = \Big( \frac{1}{2} \iint ||x-y||^2 d\rho_k(x) d\rho_k(y) \Big)^{\frac{1}{2}} = \frac{\Big(\sum_{i,j=1}^N P_i^k P_j^k ||x_i-x_j||^2\Big)^{\frac{1}{2}}}{\sqrt{2}\sum_{i=1}^N P_i^k}.$$
Then the computation of the gradient
\begin{equation}
\label{std scatter form gradient}
\frac{\partial\sigma_y}{\partial P_i^k} = \frac{\sum_{j=1}^N P_j^k ||x_i-x_j||^2 }{ \Big(2\sum_{i,j=1}^N P^j_k P^l_k ||x_j-x_l||^2\Big)^{\frac{1}{2}}}
\end{equation}
involves the samples only through the pairwise distances $||x_i-x_j||^2$, which can be computed at the onset of the algorithm.
This is helpful when the data space $\mathbb{R}^d$ has very large dimension, so that computing the means $\overline{x}_k$ and distances $\|x_i-\overline{x}_k\|^2$ of (\ref{std variance form gradient}) at each iteration becomes prohibitive.
Moreover, we can replace $\|x_i-x_j\|^2$ with any ``dissimilarity measure" such as Riemannian distance, graph distance or kernel functions (even though naive substitution might not be justified by our barycenter model).
Nevertheless, computing (\ref{std scatter form gradient}) takes $O(N^2)$ time, while (\ref{std variance form gradient}) takes $O(N\cdot d)$, so the latter is more efficient for large sample sets of low-dimensional data.
\end{remark}

In the isotropic scenario, barycentric clustering (Algorithms \ref{alg: barycentric clustering soft} and \ref{alg: barycentric clustering hard}) can be modified via (\ref{std variance form gradient}) into the following:\\

\begin{algorithm}[H]
Initialize the means $\{\overline{x}_k\}$ randomly and the stochastic matrix $P$ (by the closest $\overline{x}_k$)\;
\While{not converging}{
   	Compute and normalize the gradient $\nabla_{P}Tr[\Sigma_y] = \sum_{i,k} \big(\sigma_k+\frac{||x_i-\overline{x}_k||^2}{\sigma_k}\big) \e_{ik}$\\
    Optimal step size $\eta$ by backtracking\\
    Update $P \gets Proj_{\prod_i\Delta^K_i}\big({P} -\eta \cdot\nabla_{P}Tr[\Sigma_y]\big)$
    
    Update the cluster means $\overline{x}_k$ and standard deviations $\sigma_k$
}
\Return{\text{\upshape Assignment} $P$}
\caption{Isotropic barycentric clustering}
\label{alg: barycentric clustering isotropic soft}
\end{algorithm}

\medskip

\begin{algorithm}[H]
Initialize the means $\{\overline{x}_k\}$ (randomly) and labels $k_i$ (by the closest $\overline{x}_k$)\;
\While{not converging}{
    Update cluster means $\overline{x}_k$ and standard deviations $\sigma_k$\\
  \For{$x_i$ in sample}{
   	$k_i \gets \text{argmin}_{k}\big(\frac{||x_i - \overline{x}_k||^2}{\sigma_k} + \sigma_k \big)$\;
   }
}
\Return{\text{\upshape Labels} $\{k_i\}$}
\caption{Barycentric $k$-means}
\label{alg: barycentric clustering isotropic hard}
\end{algorithm}

\noindent
We name Algorithm \ref{alg: barycentric clustering isotropic hard} ``Barycentric $k$-means", as it closely resembles $k$-means.

As discussed in Section \ref{sec: gradient descent solution}, numerical errors could arise if some cluster degenerates into a point and its standard deviation $\sigma_k$ vanishes. In practice, one can replace the denominator $\sigma_k$ in the gradient (\ref{std variance form gradient}) by $\sigma_k + \epsilon$ for some small $\epsilon$.

Section \ref{sec: performance} will confirm the expectation that these algorithms are more robust than $k$-means under varying proportions and radii ($P_k$ and $\sigma^2_k$), but are more vulnerable to non-isotropy ($\Sigma_k$ not of the form $\sigma^2_k I$) than Algorithms \ref{alg: barycentric clustering soft} and \ref{alg: barycentric clustering hard}.

\subsection{Relation to Mahalanobis distance}
\label{sec: related work}

Barycentric clustering is not the first clustering algorithm that deals with non-isotropic clusters using second moment information. A series of clustering methods \cite{chernoff1979metric,gustafson1979fuzzy,krishnapuram1999note} based on $k$-means measure the distance between sample points and clusters by the Mahalanobis distance:
\begin{equation}
\label{Mahalanobis distance}
d^2(x_i,\overline{x}_k) = (x_i-\overline{x}_k)^T \Sigma_k^{-1}(x_i-\overline{x}_k),
\end{equation}
which reduces the distance along the directions corresponding to the large eigenvalues of the covariance $\Sigma_k$. However, as pointed out in \cite{krishnapuram1999note},  applying (\ref{Mahalanobis distance}) to $k$-means has the problem that the objective function (\ref{SSE}) becomes trivial:
$$SSE = \sum_{i,k} P_k^i (x_i-\overline{x}_k)^T \Sigma_k^{-1}(x_i-\overline{x}_k) = \sum_k P_k Tr[\Sigma_k\Sigma_k^{-1}] \equiv Tr[I].$$
The Gustafson–Kessel algorithm \cite{gustafson1979fuzzy,krishnapuram1999note} remedies this problem by modifying (\ref{Mahalanobis distance}) into
\begin{equation}
\label{Mahalanobis distance modified}
d^2(x_i,\overline{x}_k) = \text{det}(\Sigma_k)^{\frac{1}{d}} (x_i-\overline{x}_k)^T \Sigma_k^{-1}(x_i-\overline{x}_k).
\end{equation}

To compare barycentric clustering and the Mahalanobis distance-based algorithms, note that (\ref{Mahalanobis distance}) is dimensionless, in the sense that any shrinkage or dilation of cluster $\rho_k$ (with respect to the mean $\overline{x}_k$) would be completely cancelled out in (\ref{Mahalanobis distance}), which explains how its objective function becomes trivial. Meanwhile, both the squared Euclidean distance and the modified Mahalanobis distance (\ref{Mahalanobis distance modified}) have dimension $[l]^2$, where $[l]$ denotes the unit of length. For our algorithms, if we assume that the weight $P_k$ is small so that the influence of $\Sigma_k$ on $\Sigma_y$ is small, then it is routine to check that the gradient (\ref{gradient of objective}, \ref{weight matrix}) has dimension $[l]$.

This comparison becomes explicit in the isotropic setting:
$$\frac{||x_i-\overline{x}_k||^2}{\sigma^2_k}, ~||x_i-\overline{x}_k||^2, ~\frac{||x_i-\overline{x}_k||^2}{\sigma_k}+\sigma_k$$
The first term is the Mahalanobis distance (\ref{Mahalanobis distance}), the second term is squared Euclidean distance or equivalently the modified Mahalanobis (\ref{Mahalanobis distance modified}), and the third term is the gradient (\ref{std variance form gradient}), the isotropic version of (\ref{gradient of objective}). Hence, our algorithms can be seen as a balanced solution between the Euclidean case that completely ignores second moment information and the Mahalanobis case where too much normalization nullifies the problem.

As a side remark, the EM algorithm \cite{redner1984mixture,alpaydin2009introduction} is also a clustering method that can recognize non-isotropic clusters, typically modeling $\rho_k$ as Gaussian distributions. The Gaussians are essentially an exponential family built from the Mahalanobis distance (\ref{Mahalanobis distance}), indicating a possible connection between EM and our gradients (\ref{std variance form gradient}, \ref{gradient of objective}), and thus a connection between maximum-likelihood-based method and the Wasserstein barycenter framework.

\subsection{An alternative approach}
\label{sec: alternative approach}

We describe briefly here a different approach to solving the clustering problem (\ref{objective: clustering}), which avoids directly computing the barycenter $\mu$. By the Variance Decomposition Theorem of \cite{yang2019BaryNet}, we have the identity
\begin{equation*}
Tr[\Sigma_x] = Tr[\Sigma_y] + \int W_2^2\big(\rho(\cdot|z),\mu\big) d\nu(z) ,
\end{equation*}
where $\Sigma_x,\Sigma_y$ are the covariances of the data $\rho$ and the barycenter $\mu$. Since $\Sigma_x$ is fixed, the clustering problem (\ref{objective: clustering}) that minimizes $Tr[\Sigma_y]$ is equivalent to the following maximization problem
\begin{equation}
\label{total transport cost objective}
\max_{\nu(\cdot|x)} \int W_2^2\big(\rho(\cdot|z),\mu\big) d\nu(z) = \max_{P} \sum_{k=1}^K P_k W_2^2(\rho_k,\mu)
\end{equation}
where as usual, $P = [P_i^k]$ denotes the assignment distribution $P^i_k = \nu(k|x_i)$ and the clusters $\rho_k$ are determined by (\ref{compute clusters}).

Now we seek to remove $\mu$ from (\ref{total transport cost objective}). It is an interesting observation that the term $\sum_k P_k W_2^2(\rho_k,\mu)$ appears like a ``variance", the weighted mean square distance between each ``point" $\rho_k$ and their ``mean" $\mu$. Recall that we have the following identity:
\begin{equation*}
\forall \rho \in \mathcal{P}_2(\mathbb{R}^d), ~\int ||x-\overline{x}||^2 d\rho(x) = \frac{1}{2} \iint ||x-y||^2 d\rho(x)d\rho(y)
\end{equation*}
which computes the variance without involving the mean $\overline{x}$. Similarly, the following result allows us to compute (\ref{total transport cost objective}) without $\mu$.

\begin{proposition}
\label{thm: scatter formulation of total transport cost}
For any measurable space $Z$ and joint distribution $\pi$ from (\ref{joint distribution}), assume that $\rho$ has finite second moments and all conditionals $\rho(\cdot|z)$ are Gaussian. Let $\mu$ be any Wasserstein barycenter of $\pi$. Then, we have the following identity:
\begin{equation}
\label{scatter formulation of total transport cost}
\int W_2^2\big(\rho(\cdot|z),\mu\big) d\nu(z) = \frac{1}{2} \iint W_2^2\big(\rho(\cdot|z_1),\rho(\cdot|z_2)\big) d\nu(z_1)d\nu(z_2)
\end{equation}
\end{proposition}

\begin{proof}
See Appendix \ref{appendix: scatter formulation}.
\end{proof}

It follows that the clustering problem (\ref{objective: clustering}) is equivalent to
\begin{equation}
\label{objective: clustering pairwise}
\max_{P} \frac{1}{2} \sum_{k,h=1}^K P_k P_h W_2^2(\rho_k,\rho_h).
\end{equation}
By Theorem 2.1 of \cite{cuesta1996bounds}, as the clusters $\rho_k$ are assumed Gaussian, the cost terms $W_2^2(\rho_k,\rho_h)$ can be computed via
\begin{equation}
\label{location-scale transport cost}
W_2^2(\rho_k,\rho_h) = ||\overline{x}_k-\overline{x}_h||^2 + Tr[\Sigma_k]+Tr[\Sigma_h]-2Tr\big[(\Sigma_h^{\frac{1}{2}}\Sigma_k\Sigma_h^{\frac{1}{2}})^{\frac{1}{2}}\big]
\end{equation}
Lemma 2.4 of \cite{rippl2016limit} provides a formula for the partial derivatives of (\ref{location-scale transport cost}) with respect to $\overline{x}_k,\Sigma_k$. Hence one could optimize problem (\ref{objective: clustering pairwise}) using gradient descent.

Although this approach avoids computing the iteration (\ref{iteration scheme}), a drawback is that its objective function contains $O(K^2)$ terms as we are solving pairwise distances.

\section{Continuous extension}
\label{sec: continuous clustering}

Having studied the simplest case with discrete latent space $\ZS = \{1,\dots K\}$ and the resulting clustering problem (\ref{objective: clustering}), this section extends the analysis to continuous $\ZS$ and derives novel algorithms that construct principal curves and surfaces.

For brevity, we focus on the simple case when $\ZS=\R$ and all clusters $\rho(\cdot|z)$ are isotropic, as it is straightforward to extend the discussion to more general cases.
Hence, given a joint distribution $\pi=\rho(\cdot|z)\nu$, we assume that all conditionals are isotropic Gaussians $\rho(\cdot|z) = \N(\overline{x}(z),\sigma^2(z)I_d)$ with means $\overline{x}(z)$ and variance $\sigma^2(z)$ (in each dimension). Then, their barycenter $\mu$ is given by Corollary \ref{cor: isotropic std}.

In practice, we are only given a finite sample set $\{x_i\}_{i=1}^N \subseteq \R^d$. If one uses hard assignment $\nu(\cdot|x_i) = \delta_{z_i}$ and model each label $z_i \in \R$ as an independent variable to be optimized, it would be impossible to evaluate the barycenter's variance (\ref{average std}): since the sample set is finite whereas there are infinitely many $z \in \R$, almost all conditionals $\rho(\cdot|z)$ will be represented by zero or at most one sample point.

Our solution is inspired by human vision. For the image below, it is evident that $\rho(\cdot|z_1)$ has greater variance than $\rho(\cdot|z_0)$, even though there is no sample point whose assignment is exactly $z_0$ or $z_1$. The key is that we can estimate $\rho(\cdot|z_1)$ using the points nearby, $\{x_i | z_i \approx z_1\}$.

\begin{figure}[H]
\centering
\subfloat{\includegraphics[scale=0.65]{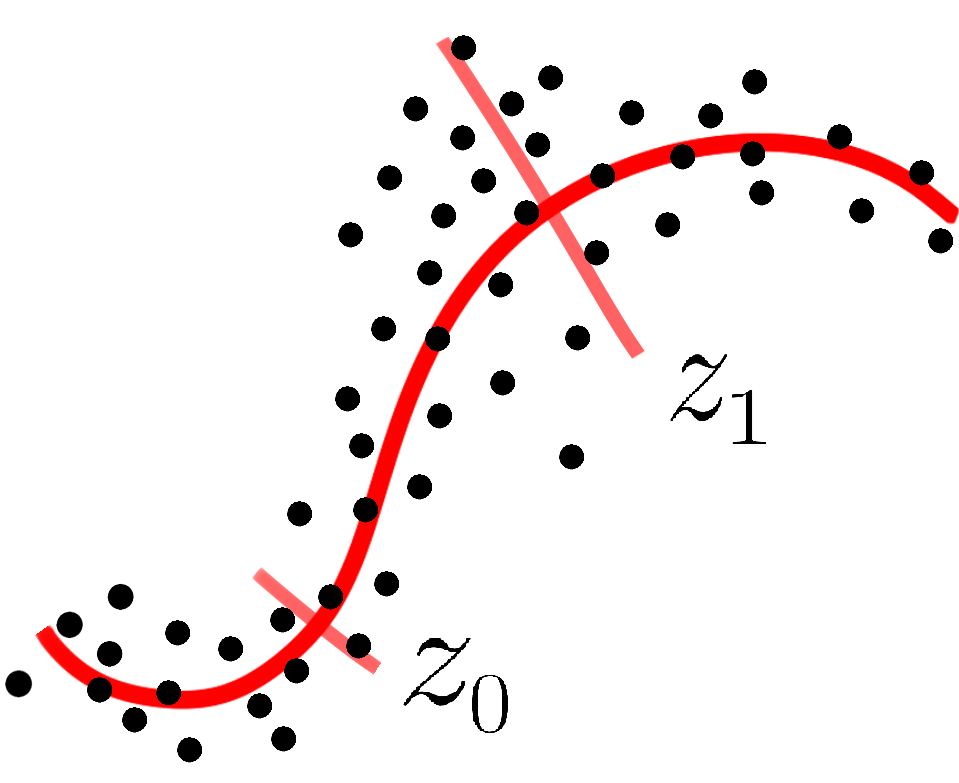}}
\hspace{1mm}
\subfloat{\includegraphics[scale=0.65]{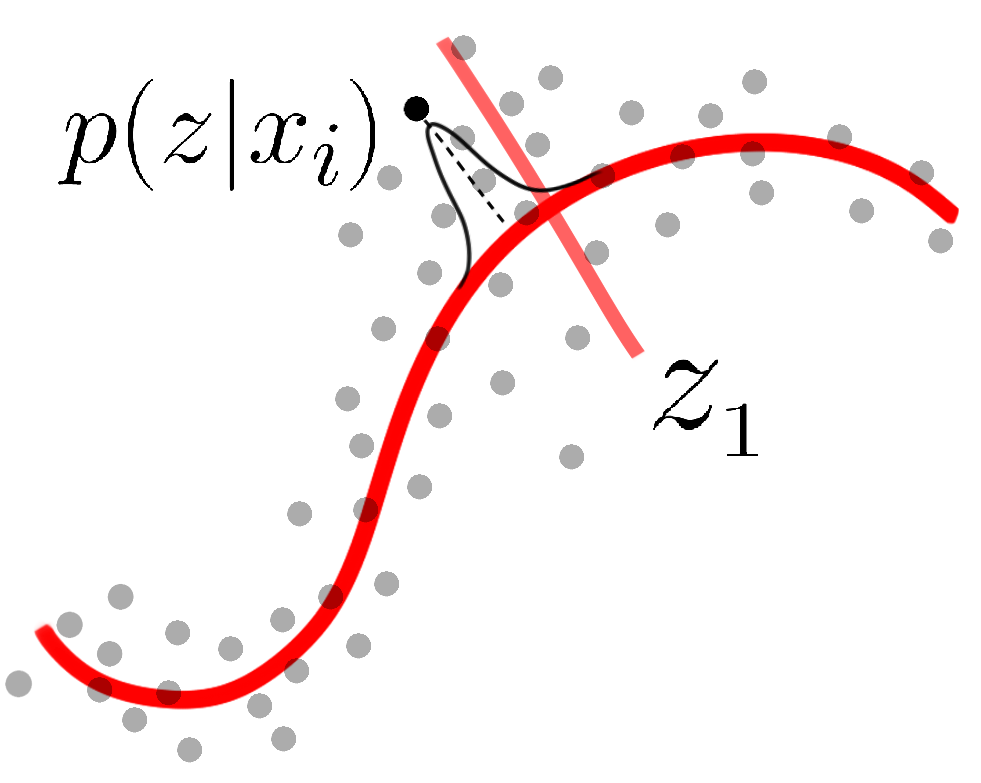}}
\caption{A two-dimensional sample set. Left: assignment of $z$ by orthogonal projection onto the red curve. Right: soft assignment $\nu(\cdot|x_i)$.}
\label{fig: soft labeling}
\end{figure}

\noindent
Hence, it is natural to use a soft assignment $\nu(\cdot|x_i)$ on $\ZS=\R$, which is concentrated around some $z_i$ and decays for $z$ far away from $z_i$. Effectively, the latent distribution $\{z_i\}$ is smoothed into
$$\nu = \frac{1}{N}\sum_{i=1}^N \nu(\cdot|x_i)$$
Given any $z$, the conditional density $\rho(x|z)$ can be estimated using Bayes' formula,
\begin{align}
\label{Bayes conditional density}
\rho(x_i|z) &= \frac{\nu(z|x_i)\rho(x_i)}{\nu(z)} = \frac{\nu(z|x_i)}{\sum_{j=1}^N \nu(z|x_j)},
\end{align}
and Corollary \ref{cor: isotropic std} implies that the barycenter's standard deviation $\sigma$ is given by
\begin{equation}
\label{std variance form continuous}
\sigma = \int \sigma(z) d\nu(z) = \int \Big[\frac{\sum_{i=1}^N ||x_i-\overline{x}||^2 \nu(z|x_i)}{\sum_{i=1}^N \nu(z|x_i)} \Big]^{\frac{1}{2}} d\nu(z).
\end{equation}
The objective of clustering (\ref{objective: factor discovery with barycenter}) is now equivalent to
\begin{equation}
\label{clustering objective continuous}
\min_{\nu(\cdot|x)} \sigma = \min_{\nu(\cdot|x)} \int \Big[\frac{\sum_{i=1}^N ||x_i-\overline{x}||^2 \nu(z|x_i)}{\sum_{i=1}^N \nu(z|x_i)} \Big]^{\frac{1}{2}} d\nu(z)
\end{equation}

For simplicity, we parameterize the assignment distribution of each sample $x_i$ by a one-dimensional Gaussian, $\nu(\cdot|x_i) = \N(\overline{z}_i,\epsilon^2)$, where $\overline{z}_i$ are the means and $\epsilon^2$ is the common variance. The means $\overline{z}_i$ are parameters by which we minimize (\ref{std variance form continuous}), and we define the vector $\overline{z} = [\overline{z}_i]$. When the sample set $\{x_i\}$ is large, rather than having an independent variable $\overline{z}_i$ for each $x_i$, one can replace $\overline{z}_i$ by a parameterized function $\overline{z}(x_i)$, for instance a neural network.

Note that $\epsilon^2$ should adapt to the set $\{\overline{z}_i\}$. Otherwise, a fixed $\epsilon^2$ would lead to the trivial solution where the $\overline{z}_i$ are arbitrarily far apart, so that by (\ref{Bayes conditional density}) each conditional $\rho(\cdot|z)$ would be concentrated at some $x_i$ and $\sigma(z)$ would go to zero.
Hence, we choose $\epsilon^2$ so as to make the distributions $\nu(\cdot|x_i)$ close to each other for nearby $x_i$. Intuitively, for some fixed $0 < \alpha < 1$ (e.g. $\alpha = 10\%$), we want that for each $\rho(\cdot|z)$, roughly a fraction $\alpha$ of the $\{x_i\}$ participate in $\rho(\cdot|z)$. The trivial solution corresponds to $\alpha \approx 0$, and one should not set $\alpha \approx 1$ either, for otherwise each $x_i$ would have significant presence in each $\rho(\cdot|z)$, contrary to the goal of clustering as a partition of $\{x_i\}$. Hence we set the following objective for $\epsilon^2$, based on maximum likelihood:
$$\max_{\epsilon^2} \prod_{i,j=1}^N \N \big(\overline{z}_j ~\big|~ \overline{z}_i,(\epsilon/\alpha)^2\big)$$
%
%
with optimal solution given by
\begin{equation}
\label{label variance def variance}
\epsilon^2 = \frac{\alpha^2}{2N^2} \sum_{i,j}^N ||\overline{z}_i-\overline{z}_j||^2 = \frac{\alpha^2}{N} \sum_{i=1}^N ||\overline{z}_i - \overline{\overline{z}}||^2 = \alpha^2 Var(\{\overline{z}_i\}),
\end{equation}
where $\overline{\overline{z}}$ is the sample mean. This choice of $\epsilon^2$ dilates $\nu(\cdot|x_i)$ proportionally to the spread of the $\{\overline{z}_i\}$, thus preventing the trivial solution.

To fix $\overline{z}_i$, we can further require that their mean should not drift away from $0$. Adding this an extra term $\overline{\overline{z}}^2$ to the penalty yields the simpler formula
\begin{equation}
\label{label variance def L2}
\epsilon^2 = \alpha^2 (Var(\{\overline{z}_i\}) + \overline{\overline{z}}^2) = \alpha^2 \frac{\|\overline{z}\|^2}{N},
\end{equation}
%
so we propose
\begin{equation}
\label{soft assignment over real line}
\nu(\cdot|x_i) = \mathcal{N} \Big( \overline{z}_i, \frac{\alpha^2 ||\overline{z}||^2}{N} \Big).
\end{equation}

\medskip
Next, we derive the gradient of the barycenter's standard deviation $\sigma$. By Appendix \ref{appendix: continuous gradient}, we can differentiate under the integral sign in (\ref{std variance form continuous}) to obtain the following gradient:
\begin{align}
\label{std expectated gradient}
\frac{\partial \sigma}{\partial \overline{z}_i} &= \mathbb{E}_{v}[G_i(z)]\\
\label{std gradient continuous pointwise}
G_i(z) &= \frac{1}{2N\cdot \nu(z)} \Big[ C(z) \overline{z}_i + \frac{z-\overline{z}_i}{\epsilon^2} \nu(z|x_i) \Big(\sigma(z)+\frac{||x_i-\overline{x}(z)||^2}{\sigma(z)} \Big) \Big] \\
\label{std gradient continuous constant}
C(z) &= \frac{1}{||\overline{z}||^2}\sum_{j=1}^N \Big[\sigma(z)+\frac{||x_j-\overline{x}(z)||^2}{\sigma(z)}\Big] \cdot \Big[ \frac{||z-\overline{z}_j||^2}{\epsilon^2} - \frac{1}{||\overline{z}||} \Big] \rho(z|x_j)
\end{align}
The computation of the integrand $G_i(z)$ for any $z$ takes linear time $O(N)$.
Estimating the expectation (\ref{std expectated gradient}) by random sampling from the latent distribution $\nu$, we obtain the following stochastic gradient descent algorithm,
which will be tested in Section \ref{sec: seismic test}:

\begin{algorithm}[H]
\textbf{Input:} Sample $\{x_i\}_{i=1}^N$, learning rate $\eta$, proportion constant $\alpha$.\\
Initialize each $\overline{z}_i$ either randomly in $[-1,1]$ or proportionally to the principal component of the sample.\\
\While{not converging}{
Randomly sample a latent variable $z$ from $\nu = \frac{1}{N} \sum_{i=1}^N \nu(\cdot|x_i)$.\\
Compute the conditional mean $\overline{x}(z)$ and standard deviation $\sigma(z)$\\
Compute the constant $C(z)$ by (\ref{std gradient continuous constant})\\
Update each $\overline{z}_i$ by the gradient (\ref{std gradient continuous pointwise}): $\overline{z}_i \gets \overline{z}_i - \eta G_i(z)$
}
\Return{$\overline{z}_i$}
\caption{Affine Factor Discovery. It can be seen as a continuous version of Algorithm \ref{alg: barycentric clustering isotropic soft}.}
\label{alg: continuous barycentric clustering}
\end{algorithm}

\begin{remark}
\label{remark: principal curve}
Whenever we obtain a joint distribution $\pi = \nu(\cdot|x)\rho$, the conditional mean $\overline{x}(z)$ of $\rho(\cdot|z)$ can be seen as a curve, parameterized by $z\in\mathbb{R}$, that summarizes the data $\rho$. If, as in Remark \ref{remark: location family}, we make the simplifying assumption that all conditional $\rho(\cdot|z)$ are equivalent up to translations, then the variance of the barycenter can be computed by
\begin{equation*}
\forall z, ~\sigma^2 = \sigma^2(z) \longrightarrow \sigma^2 = \int \sigma^2(z) d\nu(z) \approx \frac{1}{N}\sum_{i=1}^N ||x_i-\overline{x}(z_i)||^2
\end{equation*}
Or, in terms of soft assignments from (\ref{std variance form continuous}),
\begin{equation*}
\sigma^2 = \frac{1}{N}\sum_{i=1}^N \int ||x_i-\overline{x}(z_i)||^2 d\nu(z|x_i)
\end{equation*}
Since the factor discovery (\ref{objective: expectation variance}) is to minimize $\sigma^2$, an immediate corollary is that, given any curve $\overline{x}(z)$, the sample $x_i$ should be assigned to the closest point on $\overline{x}(z)$:
\begin{equation*}
z_i = \text{argmin}_z ||x_i - \overline{x}(z) ||^2
\end{equation*}
or the soft assignment $\nu(\cdot|x_i)$ should be concentrated around this $z_i$. It follows that our formulation of clustering is reduced to an alternating descent algorithm that alternates between updating the conditional means $\overline{x}(z)$ and reassigning $z_i$. Yet, this procedure is exactly the principal curve algorithm \cite{hastie1989principal,hastie2005elements}. Hence, problem (\ref{clustering objective continuous}) is a generalization of principal curves (and principal surfaces if we set $\ZS=\R^k$).
\end{remark}

\section{Performance}
\label{sec: performance}

Sections \ref{comparison of soft assignments} and \ref{comparison of hard assignments} compare $k$-means and the barycentric clustering algorithms on artificial data that deviate from the ``standard data'' of $k$-means, and Section \ref{sec: performance on real-world data} tests them on real-world classification data sets. Finally, Section \ref{sec: seismic test} tests Affine factor discovery (Algorithm \ref{alg: continuous barycentric clustering}) on earthquake data to uncover meaningful continuous latent variables.

\subsection{Comparison of soft assignments}
\label{comparison of soft assignments}

We design two families of artificial data. The ``expansion test'' is a collection of three spherical Gaussian distributions in $\mathbb{R}^2$:\\
\centerline{100 samples from $\mathcal{N}\big([0,0]^T
, \frac{1}{10}I \big),$}\\
\centerline{100$(1+t)$ samples from $\mathcal{N}\big([0,2+t]^T, \frac{(1+t)^2}{10}I \big),$}\\
\centerline{100$(1+2t)$ samples from $\mathcal{N}\big(\big[\frac{t+1}{t+2}\sqrt{12(2t+1)}, \frac{2(1-t^2)}{t+2}\big]^T, \frac{(1+2t)^2}{10}I \Big).$}\\
The means and variances are designed so that, for all $t\geq 0$, the three samples are roughly contained in three pairwise adjacent balls of radii $1$, $1+t$ and $1+2t$. As $t$ increases, the sample sizes and radii grow in distinct rates.
The ``dilation test'' is given by\\
\centerline{100 samples from $\mathcal{N}\Big(\begin{bmatrix} 0\\1\end{bmatrix}, \frac{1}{25}\begin{bmatrix} (1+t)^2 & 0\\ 0 & 1\end{bmatrix} \Big),$}\\
\centerline{100 samples from $\mathcal{N}\Big(\begin{bmatrix} 0\\0\end{bmatrix}, \frac{1}{25}\begin{bmatrix} 1 & 0\\ 0 & 1\end{bmatrix} \Big),$}\\
\centerline{100 samples from $\mathcal{N}\Big(\begin{bmatrix} 0\\ -1\end{bmatrix}, \frac{1}{25}\begin{bmatrix} (1+t)^2 & 0\\ 0 & 1\end{bmatrix} \Big),$}\\
which are Gaussians stretched horizontally at different rates.
The expansion test challenges the ``standard data'' (Section \ref{sec: relation to k-means}) in its first two assumptions: similar radii and similar proportions, while the dilation test challenges the assumption of isotropy. In both cases, the amount of deviation from the standard data is parametrized by $t\ge 0$.

The performance of each algorithm is measured by its correctness rate, the percentage of overlap between the true labeling and the labeling produced by the algorithm, maximized over all identifications between the proposed clusters and the true clusters: given the true labeling $\{z_i\}$ and either the labeling $\{k_i\}$ or the stochastic matrix $P$ produced by algorithm, we define the correctness rate as
\begin{equation}
\label{correct rate}
\max_{g \in S_K} \sum_{i} \mathbf{1}_{z_i = g(k_i)} \text{ or } \max_{g\in S_K}\sum_{i}P^i_{g(z_i)}
\end{equation}
where $g$ ranges over the permutation group $S_K$.\\

We first compare the soft assignment algorithms: Fuzzy $k$-means, barycentric clustering (Algorithm \ref{alg: barycentric clustering soft}), and isotropic barycentric clustering (Algorithm \ref{alg: barycentric clustering isotropic soft}). Note that $k$-means' objective (\ref{SSE}) is approximately a linear function in the assignment $P=[P_i^k]$, and thus the optimal solutions are the extremal points of $Dom(P) = \prod_i \Delta^K$, which are hard assignments. Hence, in order to obtain a valid comparison among soft assignments, we use the fuzzy $k$-means algorithm \cite{bezdek2013pattern}, which generalizes $k$-means, minimizing the following objective function:
$$J_c(P,\{\overline{x}_k\}) = \sum_i\sum_k (P_i^k)^c ||x_i-\overline{x}_k||^2.$$
This is a generalization of the sum of squared errors (\ref{SSE}), with an exponent $c>1$ that makes $J_c$ strictly convex in $P^i_k$, and therefore yields soft assignments. Here we adopt the common choice $c=2$.

\begin{algorithm}[H]
\KwData{Sample $\{x_i\}$, exponent $c=2$}
Initialize the means $\overline{x}_k$ and stochastic matrix $P = (P^k_i)$ randomly\\
\While{not converging}{
  \For{$x_i$ in sample and $k$ = 1 to K}{
   		$P_i^k \gets (||x_i-\overline{x}_k||^2)^{1-c}/\sum_j (||x_j-\overline{x}_k||^2)^{1-c}$
   }
  \For{$k$ = 1 to K}{
    $\overline{x}_k \gets \sum_k (P^i_k)^c x_i / \sum_k (P^i_k)^c$
  }
}
\caption{Fuzzy $k$-means}
\label{SoftClassical}
\end{algorithm}

\medskip
Each algorithm is initialized with random means $\overline{x}_k$ and assignment matrix $P$, and all algorithms share the same sample set (for each $t$). To stabilize performance, each algorithm is run 100 times over the same sample set and the result that minimizes this algorithm's objective function ($J_c$ for fuzzy $k$-means, (\ref{objective: clustering}) for barycentric clustering, and (\ref{std formula discrete}) for isotropic barycentric clustering) is selected.

The experimental results are plotted below. The first row corresponds to the expansion test with $t=2.2$, and the second row to the dilation test with $t=3.0$. The class assignment displayed is given by the maximum probability, $k_i \gets \text{argmax}_k P_k^i$.

\begin{figure}[H]
\centering
\subfloat{\includegraphics[scale=0.3]{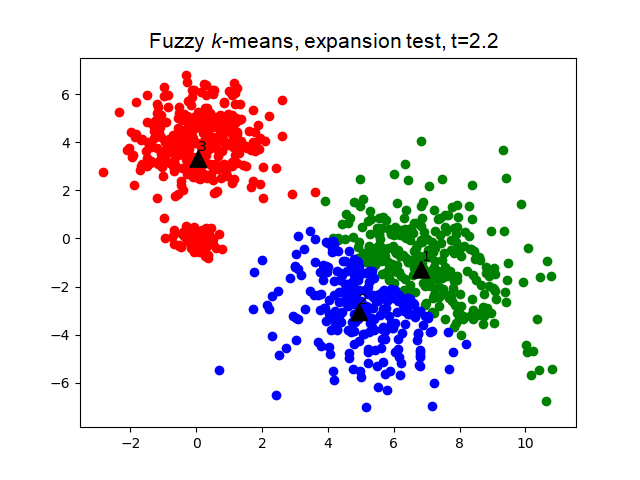}}
\subfloat{\includegraphics[scale=0.3]{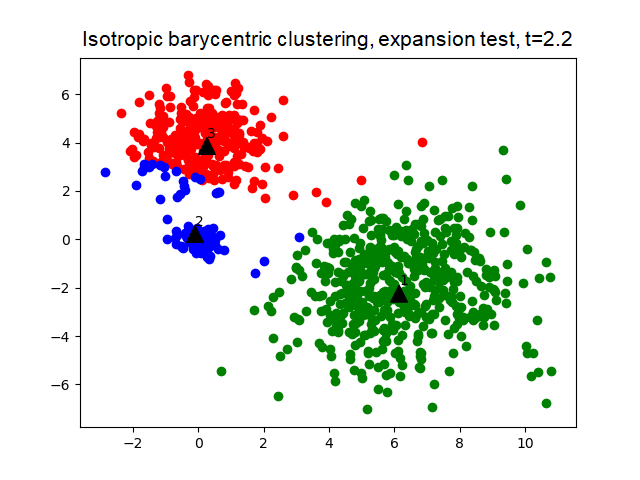}}
\subfloat{\includegraphics[scale=0.3]{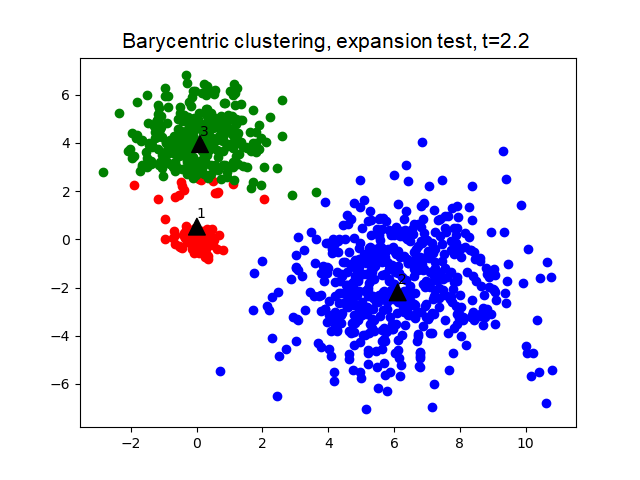}}
\vspace{-3ex}
\centering
\subfloat{\includegraphics[scale=0.3]{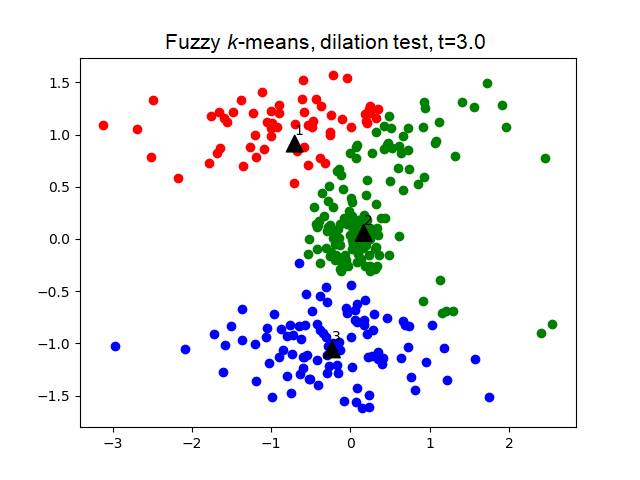}}
\subfloat{\includegraphics[scale=0.3]{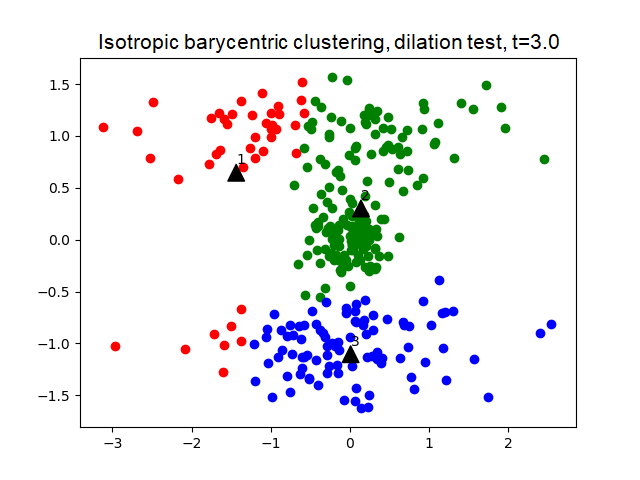}}
\subfloat{\includegraphics[scale=0.3]{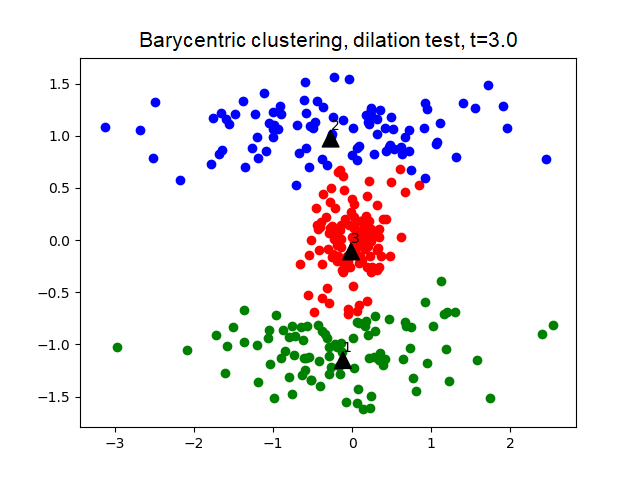}}
\caption{Clusters produced by fuzzy $k$-means (left), isotropic barycentric clustering (middle) and barycentric clustering (right). The black arrows indicate the clusters' means.}
\end{figure}
For the expansion test, fuzzy $k$-means merged the two smaller clusters and split the largest one, whereas the barycentric clustering algorithms only made a few errors on the periphery. For dilation test, the correct clusters are produced only by Barycentric clustering, whereas fuzzy $k$-means and Isotropic barycentric clustering split the clusters. These results are not surprising, since Section \ref{sec: relation to k-means} shows that $k$-means is an approximation to Barycentric clustering that assumes clusters with identical sizes and radii, while Isotropic barycentric clustering, by design, assumes isotropic clusters.

Below are the plots of correct rates (\ref{correct rate}) for $t\in [0,4]$. Fuzzy $k$-means, with a steady decline, is dominated by the barycentric clustering algorithms, while the difference between the latter two is small. Eventually, as $t\to\infty$, all algorithms deviate from the true labeling, since for very large $t$ the Gaussians become so disparate that the true labeling no longer minimizes $Tr[\Sigma_y]$ or yields reasonable clusters that agree with human perception.
\begin{figure}[H]
\centering
\subfloat{\includegraphics[scale=0.5]{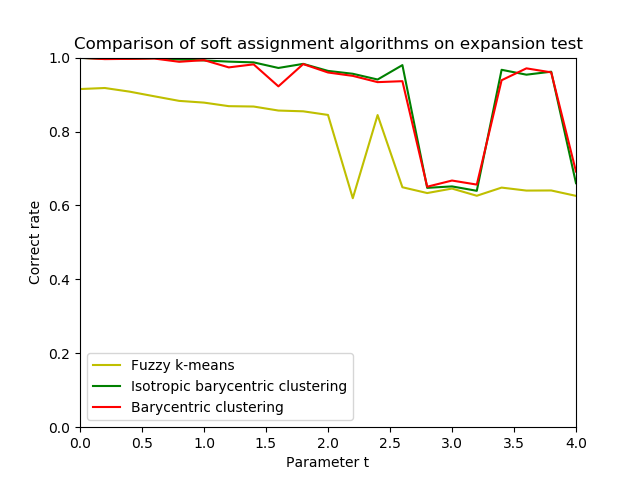}}
\subfloat{\includegraphics[scale=0.5]{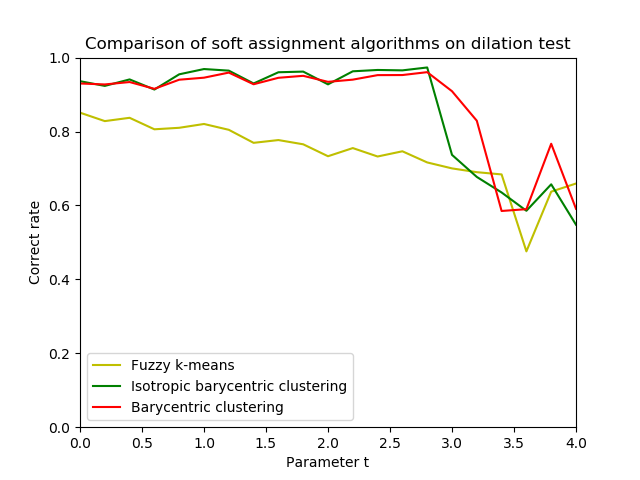}}
\caption{Correctness rates. Left: expansion test. Right: dilation test.}
\end{figure}

In fact, for the dilation test, the contrast can be seen well before $t=3.0$. The following is the result for $t=1.6$, with the shaded regions representing the convex hulls containing the ``core points'' of each class, defined as $C_k = \{x_i, P_k^i>1/3\}$. The soft clusters produced by fuzzy $k$-means exhibit significant overlap, indicating that many sample points are assigned with highly ambiguous probability vectors $P^i$.
\begin{figure}[H]
\centering
\subfloat{\includegraphics[scale=0.3]{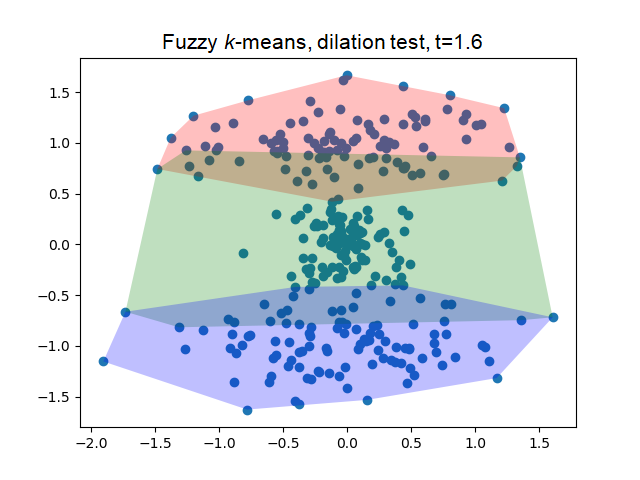}}
\subfloat{\includegraphics[scale=0.3]{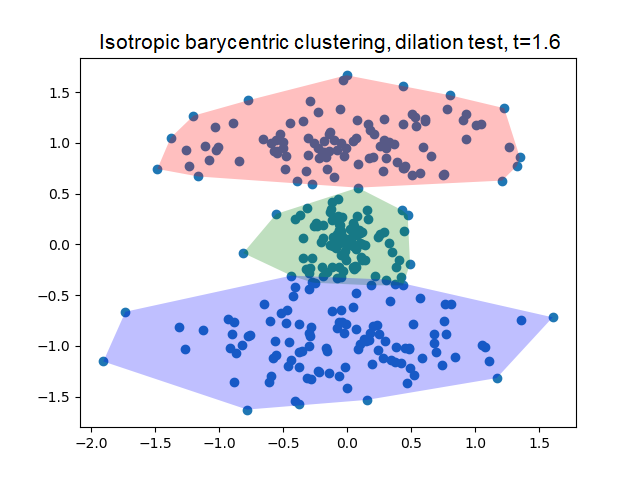}}
\subfloat{\includegraphics[scale=0.3]{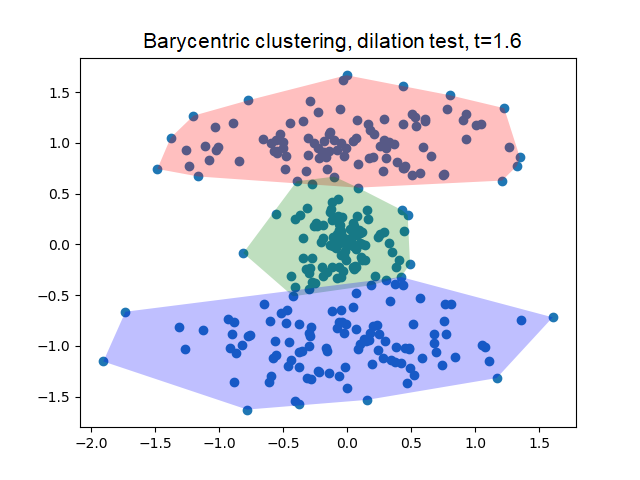}}
\caption{Convex hulls of core points. Only the barycentric clustering algorithms correctly cover each cluster.}
\label{fig: convex hull}
\end{figure}

\subsection{Comparison of hard assignments}
\label{comparison of hard assignments}
We compare next the hard assignment algorithms: $k$-means \cite{alpaydin2009introduction}, Hard barycentric clustering (Algorithm \ref{alg: barycentric clustering hard}), and Barycentric $k$-means (Algorithm \ref{alg: barycentric clustering isotropic hard}). The results on the expansion test and dilation test are plotted below. Again, for each algorithm on each sample set, the objective-minimizing result over 100 trials is selected. The performance comparison is analogous to that of soft assignment.
\begin{figure}[H]
\centering
\subfloat{\includegraphics[scale=0.33]{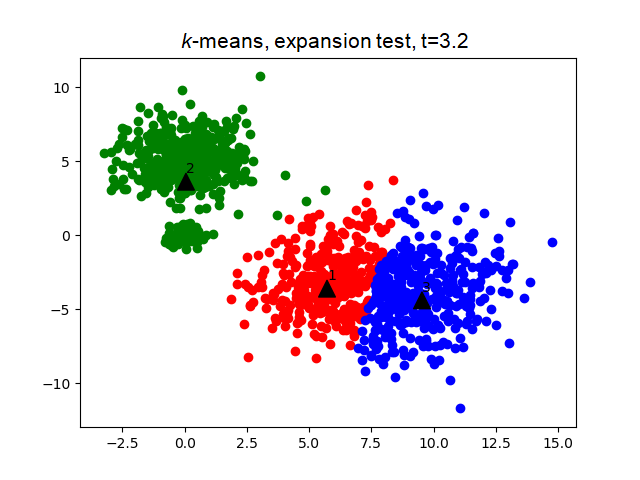}}
\subfloat{\includegraphics[scale=0.33]{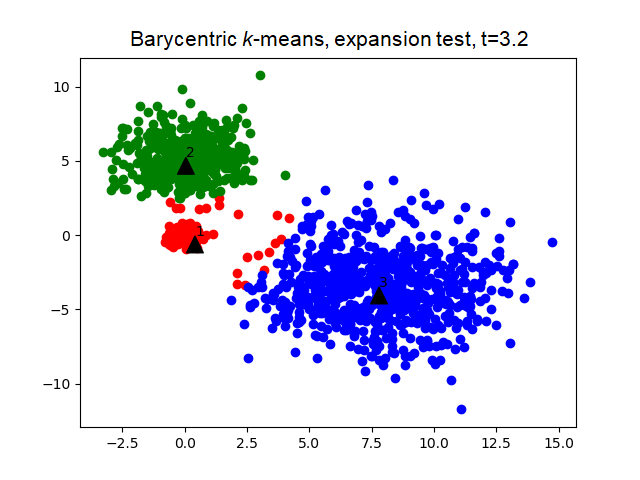}}
\subfloat{\includegraphics[scale=0.33]{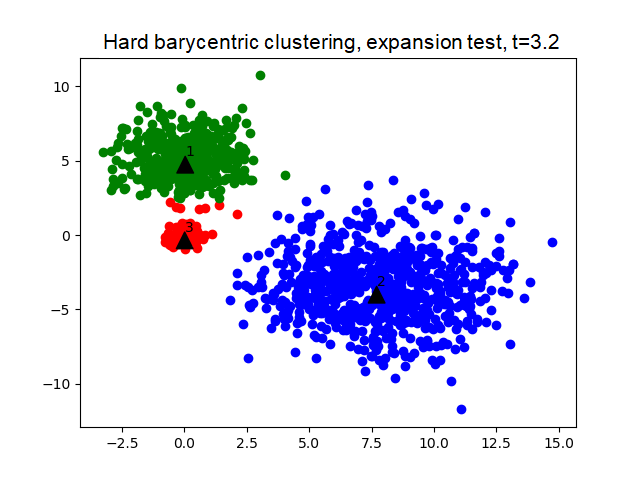}}
\vspace{-3ex}
\centering
\subfloat{\includegraphics[scale=0.33]{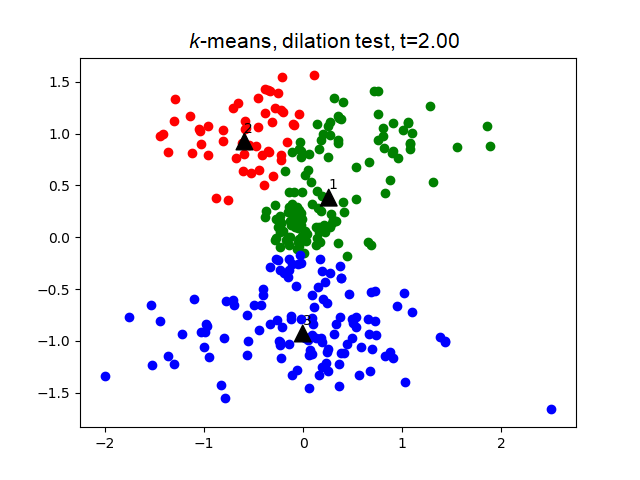}}
\subfloat{\includegraphics[scale=0.33]{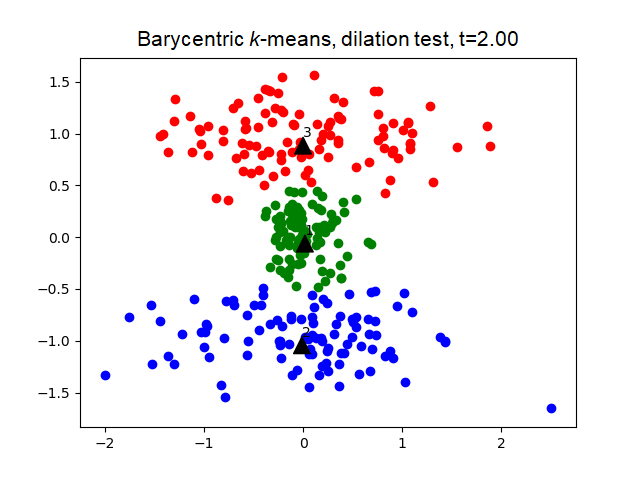}}
\subfloat{\includegraphics[scale=0.33]{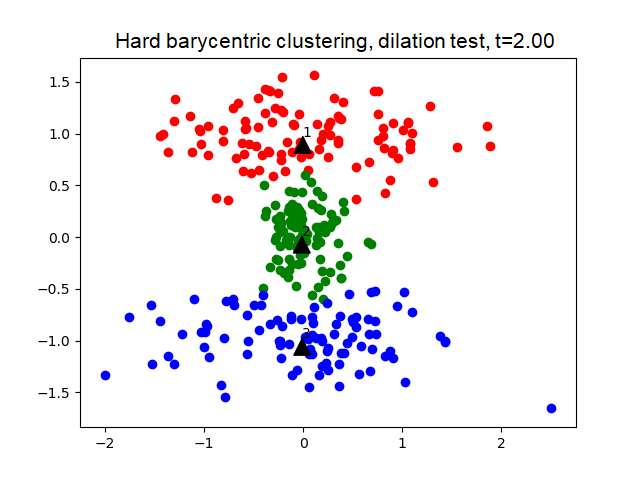}}
\caption{First row: Expansion test with $t=3.2$. Second row: Dilation test with $t=2.0$. Left: $k$-means. Middle: Barycentric $k$-means. Right: Hard barycentric clustering.}
\end{figure}

Nevertheless, Hard barycentric clustering is a simplified version of Barycentric clustering, replacing the latter's gradient descent, which moves by small steps, by class reassignment, which hops among the extremal points of $\prod_i \Delta^K$. The correctness rate curves of dilation test indicate that, whereas Barycentric $k$-means has similar performance as Isotropic barycentric clustering, Hard barycentric clustering is more unstable than Barycentric clustering.
\begin{figure}[H]
\centering
\subfloat{\includegraphics[scale=0.45]{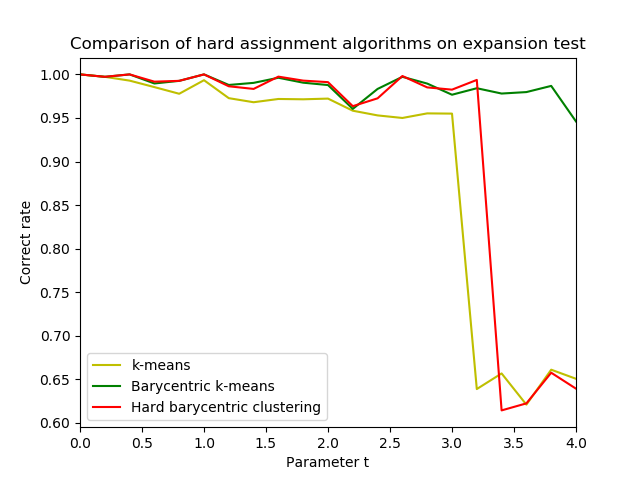}}
\subfloat{\includegraphics[scale=0.45]{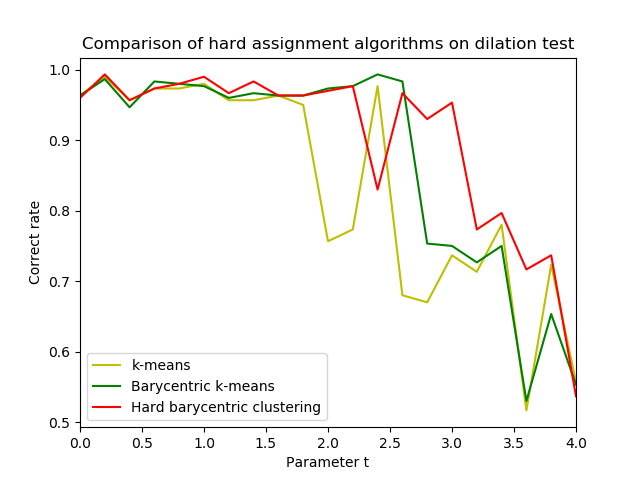}}
\caption{For the expansion test, we have the ranking: $k$-means $<$ Hard barycentric clustering $<$ Barycentric $k$-means, and for the dilation test: $k$-means $<$ Barycentric $k$-means $<$ Algorithm Hard barycentric clustering.}
\end{figure}

\subsection{Clustering on real-world data}
\label{sec: performance on real-world data}
To compare the performance of $k$-means and our algorithms on real-world problems, we use data sets from the online UCI Machine Learning Repository \cite{dheeru2017uci}. These data sets, intended for classification, are provided with labels, which we use to calculate the correctness rates (\ref{correct rate}). The ``Wine'' \cite{aeberhard1994comparative} data set classifies wines based on chemical compositions, ``Seeds'' \cite{charytanowicz2010complete} classifies wheats by the shapes of wheat kernels, ``Breast cancer (original)'' \cite{wolberg1990multisurface} classifies benign/malign cancers by the shapes of cell nuclei, while ``Breast cancer (diagnostic)'' \cite{street1993nuclear} classifies them by other cell statistics, ``Parkison's'' \cite{little2007exploiting} diagnoses the disease by the patients' voice and speech, and ``E.coli'' \cite{horton1996probabilistic} classifies proteins by their sequences and structures.

Since the setting for our clustering problem is for data in $\mathbb{R}^d$, the data's categorical attributes as well as entries with missing data are removed. The samples are normalized along each dimension before clustering, since their attributes are often on disparate scales. Again, each algorithm is run 100 times on each sample set, and the objective-minimizing result is selected.

In the following table, for each sample set, the marked entries are the ones with maximum correctness rates among the hard and soft assignment groups.

\begin{table}[H]
\centering
\begin{tabular}{lcccccc}
\cline{2-7}
\multicolumn{1}{l|}{} & \multicolumn{1}{c|}{\begin{tabular}[c]{@{}c@{}}Wine\end{tabular}} & \multicolumn{1}{c|}{\begin{tabular}[c]{@{}c@{}}Seeds\end{tabular}} & \multicolumn{1}{c|}{\begin{tabular}[c]{@{}c@{}}Breast cancer\\ (original) \end{tabular}} & \multicolumn{1}{c|}{\begin{tabular}[c]{@{}c@{}}Breast cancer\\ (diagnostic)\end{tabular}} & \multicolumn{1}{c|}{\begin{tabular}[c]{@{}c@{}}Parkinson's\end{tabular}} & \multicolumn{1}{c|}{\begin{tabular}[c]{@{}c@{}}E.coli\end{tabular}} \\ \hline
\multicolumn{1}{|l|}{Number of classes $K$} & 3 & 3 & 2 & 2 & 2 & \multicolumn{1}{c|}{8} \\ \cline{1-1}
\multicolumn{1}{|l|}{Dimension $d$} & 13 & 7 & 9 & 30 & 22 & \multicolumn{1}{c|}{6} \\ \cline{1-1}
\multicolumn{1}{|l|}{Sample size} & 178 & 210 & 683 & 569 & 197 & \multicolumn{1}{c|}{336} \\ \hline
Correct rates \% & \multicolumn{1}{l}{} & \multicolumn{1}{l}{} & \multicolumn{1}{l}{} & \multicolumn{1}{l}{} & \multicolumn{1}{l}{} & \multicolumn{1}{l}{} \\ \hline
\multicolumn{1}{|l|}{$k$-means} & 96.63 & 91.90 & 95.75 & \cellcolor[HTML]{C0C0C0}91.04 & 54.36 & \multicolumn{1}{c|}{55.65} \\ \cline{1-1}
\multicolumn{1}{|l|}{Algorithm \ref{alg: barycentric clustering isotropic hard}} & \cellcolor[HTML]{C0C0C0}97.19 & 91.90 & 96.34 & 89.46 & 53.33 & \multicolumn{1}{c|}{\cellcolor[HTML]{C0C0C0}59.82} \\ \cline{1-1}
\multicolumn{1}{|l|}{Algorithm \ref{alg: barycentric clustering hard}} & \cellcolor[HTML]{C0C0C0}97.19 & \cellcolor[HTML]{C0C0C0}92.86 & \cellcolor[HTML]{C0C0C0}96.49 & 90.69 & \cellcolor[HTML]{C0C0C0}60.00 & \multicolumn{1}{c|}{\cellcolor[HTML]{C0C0C0}59.82} \\ \hline
\multicolumn{1}{|l|}{Fuzzy $k$-means} & 60.92 & 74.76 & 87.19 & 73.79 & \cellcolor[HTML]{C0C0C0}54.21 & \multicolumn{1}{c|}{34.01} \\ \cline{1-1}
\multicolumn{1}{|l|}{Algorithm \ref{alg: barycentric clustering isotropic soft}} & \cellcolor[HTML]{C0C0C0}94.34 & \cellcolor[HTML]{C0C0C0}89.56 & \cellcolor[HTML]{C0C0C0}96.51 & 88.78 & 53.25 & \multicolumn{1}{c|}{\cellcolor[HTML]{C0C0C0}57.41} \\ \cline{1-1}
\multicolumn{1}{|l|}{Algorithm \ref{alg: barycentric clustering soft}} & 91.71 & 88.73 & 96.29 & \cellcolor[HTML]{C0C0C0}89.94 & 50.91 & \multicolumn{1}{c|}{52.67} \\ \hline
\end{tabular}
\caption{Our algorithms outperformed (fuzzy) $k$-means on the majority of data sets. For hard assignment, we have the ranking: $k$-means $<$ Barycentric $k$-means $<$ Hard barycentric clustering. For soft assignment we have: fuzzy $k$-means $<$ Barycentric clustering $<$ Isotropic barycentric clustering, although the influence of the isotropic simplification seems minuscule.}
\end{table}

One notable difference between our synthetic tests and these real-world data is that the latter have higher dimensions, which negatively influence fuzzy $k$-means' performance. Previous studies \cite{winkler2011fuzzy} have shown that, as dimension increases, the pairwise distances of the sample points become homogeneous, and fuzzy $k$-means tends to produce uniform assignments: $P^i_k \approx 1/K$. Nevertheless, the barycentric clustering algorithms remain robust, as shown below.
\begin{figure}[H]
\centering
\subfloat{\includegraphics[scale=0.275]{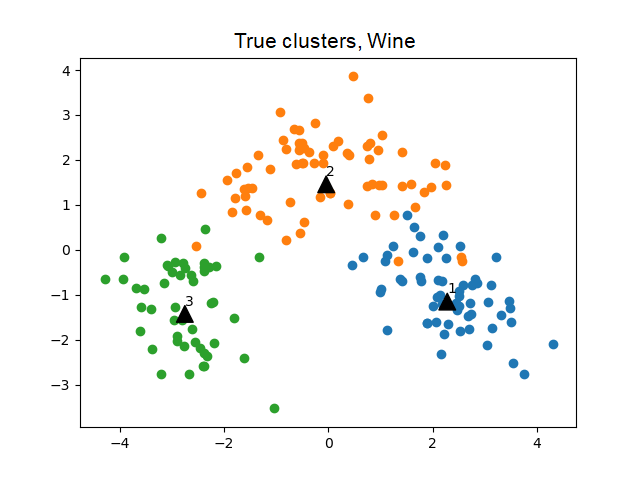}}
\subfloat{\includegraphics[scale=0.275]{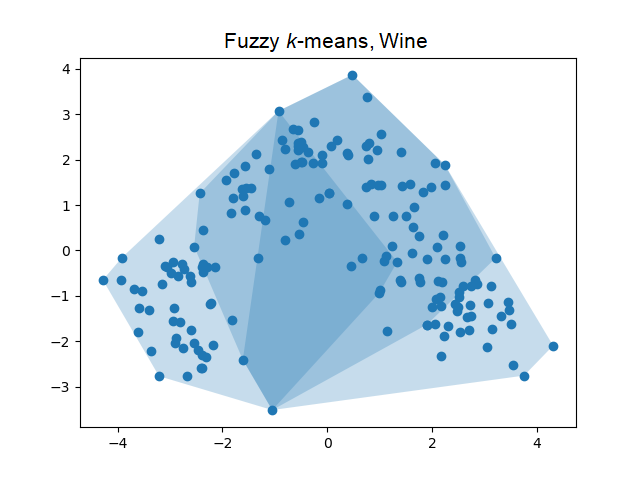}}
\subfloat{\includegraphics[scale=0.275]{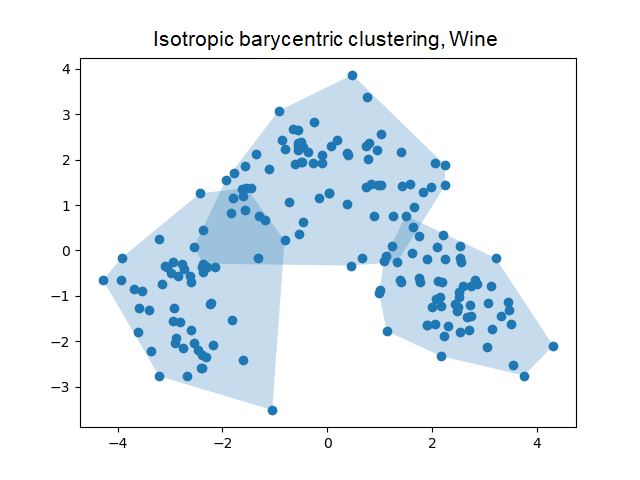}}
\subfloat{\includegraphics[scale=0.275]{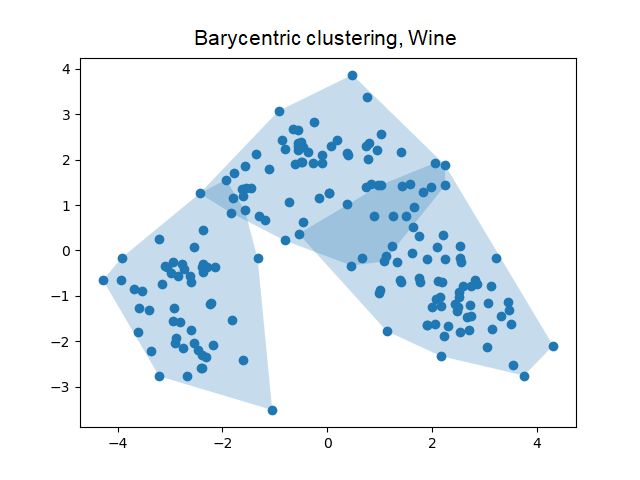}}
\caption{Soft assignment on the ``Wine'' data set. The sample is projected onto its principal 2-plane, with the true labeling given in the first panel. The following three panels correspond to fuzzy $k$-means, Isotropic barycentric clustering, and Barycentric clustering. The shaded polygons represent the convex hulls spanned by the ``core points'' of each cluster, $C_k=\{x_i, P^i_k > 1/4\}$. Fuzzy $k$-means assigned each $x_i$ with ambiguous probabilities ($P^i_k \approx 1/3$), and many sample points belong to the ``cores'' of two or more clusters, whereas the assignments produced by the barycentric clustering algorithms are relatively ``hard''.}
\end{figure}

\subsection{Continuous latent variable and seismic data}
\label{sec: seismic test}

Finally, we test Affine factor discovery (Algorithm \ref{alg: continuous barycentric clustering}). As discussed in Remark \ref{remark: principal curve}, the continuous factor discovery problem (\ref{clustering objective continuous}) generalizes principal curves, so it is natural to evaluate Algorithm \ref{alg: continuous barycentric clustering} in terms of its ``principal curve", that is, the conditional mean $\overline{x}(z)$. Given data that appears to cluster around one or several curves, the curve $\overline{x}(z)$ should discover these patterns.

We use the earthquake data from \cite{USGS2008earthquake}, which covers more than two thousand earthquakes in the 20th century in the Southeast Asia earthquake zone. The sample $\{x_i\}$ is two dimensional, recording the latitude and longitude of the earthquakes. We apply Affine factor discovery with a fixed number of iterations $T=50000$, proportion constant $\alpha=2.5\%$, and learning rate $\eta=5\times 10^{-1}$, and we initialize $\overline{z}_i$ to be proportional to longitude, which is evidently far from the optimal solution. The curve of conditional means $\overline{x}(z)$ is plotted below.
\begin{figure}[H]
\centering
\subfloat{\includegraphics[scale=0.65]{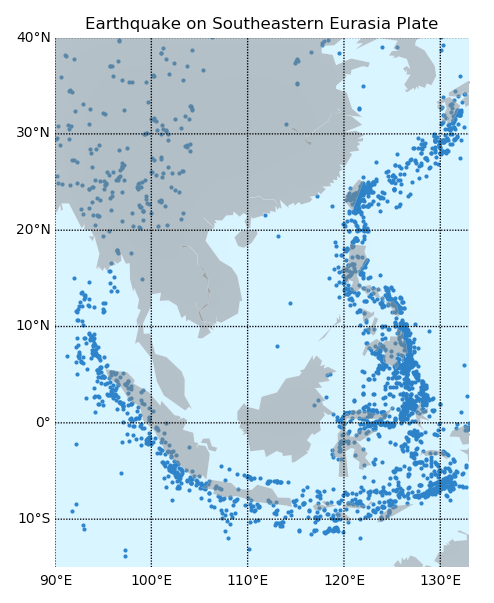}}
\subfloat{\includegraphics[scale=0.45]{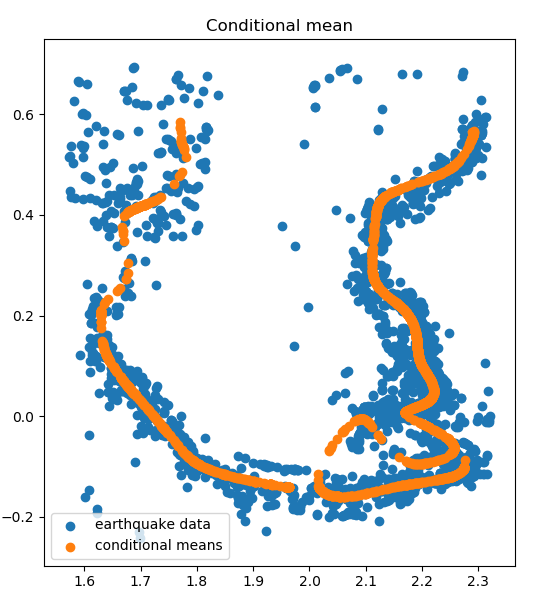}}
\caption{Left: Plot of earthquake data $\{x_i\}$. Right: The conditional means $\overline{x}(z)$ with $z$ sampled from $\nu$.}
\label{fig: earthquake test}
\end{figure}
The Southeast Asia earthquake zone lies at the intersection of the Australian Plate, Eurasian Plate, Indian Plate, Philippine Plate, Yangtze Plate, Amur Plate, and numerous minor plates and microplates. The tectonic boundaries are complex and cannot be represented by a single curve. Affine factor discovery automatically solved this problem using piecewise principal curves.
Note that even though the latent space $Z=\mathbb{R}$ is connected, the support of the latent distribution, supp$\nu$, consists of several disjoint clusters, giving rise to piecewise continuous $\overline{x}(z)$.
\begin{figure}[H]
\centering
\subfloat{\includegraphics[scale=0.8]{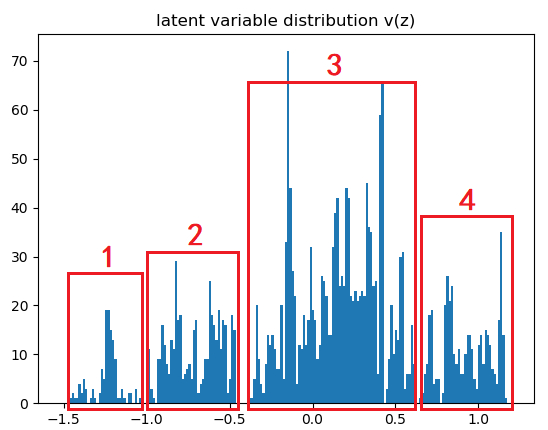}}
\subfloat{\includegraphics[scale=0.65]{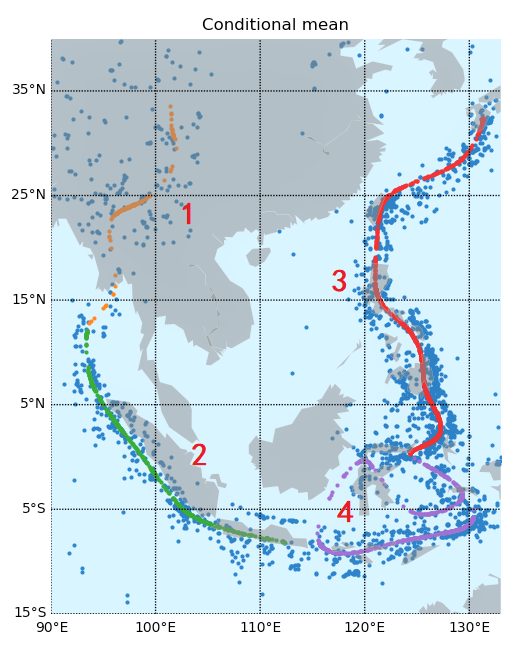}}
\caption{Left: Latent variable distribution $\nu$, represented by the conditional label means $\{\overline{z}_i\}$. We can roughly identify four disjoint components of supp$\nu$, or clusters. Right: The curve $\overline{x}(z)$ of each cluster corresponds to an earthquake belt.}
\label{fig: piecewise pricipal curve}
\end{figure}

\section{Conclusion}
\label{sec: conclusion}

Factor discovery follows the human intuition of finding significant factors that explain the data.
To pose the factor discovery problem in a systematic way, this article proposes an optimal transport based framework such that potential factors are extracted from the data and the resulting conditionals are consolidated into their barycenter, from which the procedure can continue inductively. Extending ideas proposed in \cite{tabak2018explanation}, we design a novel objective function that minimizes the unexplained variance of the data, featured as the barycenter's variance.

This article introduces several clustering and principle surface algorithms for solving the factor discovery problem in the Gaussian setting, which benefit from second moment information. Specifically,
\begin{enumerate}
\item The barycentric clustering algorithms (Algorithms \ref{alg: barycentric clustering soft} and \ref{alg: barycentric clustering hard}) leverage the clusters' covariance matrices and are capable of recognizing non-isotropic clusters with varying radii and sizes.

\item Barycentric $k$-means (Algorithms \ref{alg: barycentric clustering isotropic hard}, designed for isotropic Gaussians) runs as efficiently as $k$-means, yet with little loss in robustness compared to the general barycentric clustering algorithms.

\item Affine factor discovery (Algorithm \ref{alg: continuous barycentric clustering}), an efficient principle curve/surface algorithm that uncovers continuous latent variables, utilizes the variance of each conditional distribution.

\item All algorithms proposed reduce to the classical $k$-means and principle surface algorithm in the simplified setting when all conditionals are equal up to translations, indicating that the optimal transport framework includes the classical intra-class variance-based framework as a special case.
\end{enumerate}


The methodology developed in this article, in addition to its value as a practical tool for clustering and continuous factor discovery, opens the way to further inquiry into various directions:
\begin{enumerate}
\item Efficient non-isotropic algorithm: as discussed in Sections \ref{sec: gradient descent solution} and \ref{sec: isotropic solution}, the main difficulty in solving the general non-isotropic problem is due to the non-commutativity of the covariances $\Sigma_k$. The minimum requirement for commutativity is that all $\Sigma_k$ are simultaneously diagonalizable, that is, there exists some orthonormal basis with which all $\Sigma_k$ become diagonal matrices. Then, the variance formula of Corollary \ref{cor: isotropic std} can be applied to each coordinate, greatly simplifying the objective (\ref{objective: clustering}) and its gradients. This commutative setting is non-isotropic and more general than the isotropic setting, and the corresponding clustering algorithm might inherit both the efficiency of Algorithms \ref{alg: barycentric clustering isotropic hard} and the generality of Algorithms \ref{alg: barycentric clustering soft} and \ref{alg: barycentric clustering hard}.

\item Task-specific metrics: the Wasserstein barycenter minimizes the $L^2$ deformation of the data, so our formulations are all based on the Euclidean distance. Yet for many tasks, the most suitable underlying geometry is not Euclidean or homogeneous. For instance, in the space of images, certain deformations, such as translation and rotation, are more acceptable than others, such as tearing and smearing, even if they have similar Euclidean norm. A promising approach is to extend the algorithms to use the (squared) geodesic distance in the manifold that underlies the data, such as the Fermat distance introduced in \cite{sapienza2018weighted}.

\item Iterative factor discovery: factor discovery can be performed iteratively to uncover a hierarchy of latent factors, which could correspond to diverse latent spaces $\ZS$. It could be of practical interests to develop a principled way of conducting multi-stage factor discovery. For gene expression data, for instance, one could first apply Affine factor discovery to reduce the dimensionality from the order of millions to several hundreds, then apply Barycentric $k$-means to identify the clusters, which might correspond to meaningful demographic groups, and finally apply Affine factor discovery on the resulting barycenter to obtain continuous factors that may offer biomedical insights.

\item Curse of dimensionality: for high dimensional distributions $\rho$, training can become more difficult as the landscape contains more local minima. Moreover, as the pairwise distances between the data points become similar, the gap between the global minimum and many of the bad local minima become less prominent. One effective way to constrain and regularize the solution is to implement the assignment distribution $\nu(\cdot|x)$ or the hard assignment $z(x)$ by neural networks. This modification is straightforward for principal curve/surface, while for clustering, one can place the points $\ZS=\{1,\dots K\}$ on a grid in a low-dimensional $\R^k$ and apply a neural net that maps from $\R^k$ to $\R^d$. This construction resembles the self-organizing maps of \cite{hastie2005elements}, while the use of neural networks allow us to scale to high dimensions.

\end{enumerate}

\section*{Acknowledgments} The work of E. G. Tabak was partially supported by NSF grant DMS-1715753 and ONR grant N00014-15-1-2355.

\bibliographystyle{acm}
\bibliography{main}


\centerline{\bf \Large Appendices:}

\appendix

\section{Existence of derivative}
\label{appendix: existence of derivative}
Here we establish that $\partial \Sigma_y / \partial P_k^i$, the partial derivatives of the barycenter's covariance implicitly defined by (\ref{barycenter covariance discrete}) with respect to the assignment probabilities $P_k^i = \nu(k|x_i)$, always exist. Theorem 1.1 from \cite{del2018taylor} is a useful result on matrix derivatives which we restate below. Denote by $\mathcal{S}_d, \mathcal{S}^+_d \subseteq \mathcal{M}_d$ the linear subspace of symmetric matrices and the cone of positive-definite matrices.
\begin{theorem}
The principal matrix square root function $S\in \mathcal{S}^+_d \to S^{\frac{1}{2}}\in \mathcal{S}^+_d$ is Fréchet differentiable to any order, and the first order derivative is given by the operator $$(\nabla S^{\frac{1}{2}})(H) := \nabla S^{\frac{1}{2}}|_{H} = \int_0^{\infty} e^{-S^{\frac{1}{2}}t} \cdot H \cdot e^{-S^{\frac{1}{2}}t}dt$$
such that for any $H\in \mathcal{S}_d$ and $S+hH\in \mathcal{S}^+_d$ $$\lim_{h\to 0} \frac{1}{h}[(S+hH)^{\frac{1}{2}}-S^{\frac{1}{2}} - h(\nabla S^{\frac{1}{2}})(H)] = 0.$$
\end{theorem}

Now we prove the existence of the partial derivatives through the Implicit Function Theorem.
\begin{theorem}
For $\Sigma_1,\dots\Sigma_K \in \mathcal{S}_d^+$, the solution $\Sigma_y$ to
the covariance formula (\ref{barycenter covariance discrete}) depends differentiably on $\Sigma_1,\dots\Sigma_K$.
\end{theorem}

\begin{proof}
For convenience, define the function $$F(\Sigma_y,\Sigma_1,\dots\Sigma_K) = \sum\limits_{k=1}^K P_k (\Sigma_y^{\frac{1}{2}}\Sigma_k\Sigma_y^{\frac{1}{2}})^{\frac{1}{2}} - \Sigma_y$$
It is a composition of $C^1$
functions on $\prod_{i=1}^{K+1} S^+_d$ and thus is $C^1$. To confirm that the gradient $\nabla_{\Sigma_y}F$ is non-singular, perturb $\Sigma_y$ along an arbitrary direction $S\in \mathcal{S}_d$,
\begin{align*}
(\nabla_{\Sigma_y}F)(S) =& \sum\limits_{k=1}^K P_k \Big(\nabla\big(\Sigma_y^{\frac{1}{2}}\Sigma_k\Sigma_y^{\frac{1}{2}}\big)^{\frac{1}{2}}\Big)\Big(\big(\nabla\Sigma_y^{\frac{1}{2}}\big)(S)\Sigma_k\Sigma_y^{\frac{1}{2}} + \Sigma_y^{\frac{1}{2}}\Sigma_k\big(\nabla\Sigma_y^{\frac{1}{2}}\big)(S)\Big) - S\\
=& \sum\limits_{k=1}^K \int_0^{\infty} e^{-(\Sigma_y^{\frac{1}{2}}\Sigma_k\Sigma_y^{\frac{1}{2}})^{\frac{1}{2}}t} \Big[\Big(\int_0^{\infty} e^{-\Sigma_y^{\frac{1}{2}}u}Se^{-\Sigma_y^{\frac{1}{2}}u}du\Big)\Sigma_k\Sigma_y^{\frac{1}{2}}\\
&+\Sigma_y^{\frac{1}{2}}\Sigma_k\Big( \int_0^{\infty} e^{-\Sigma_y^{\frac{1}{2}}u}Se^{-\Sigma_y^{\frac{1}{2}}u}du \Big)\Big] e^{-(\Sigma_y^{\frac{1}{2}}\Sigma_k\Sigma_y^{\frac{1}{2}})^{\frac{1}{2}}t}dt - S
\end{align*}

To evaluate integrals of the form $\int_0^{\infty} e^{-\Sigma^{\frac{1}{2}}t}S e^{-\Sigma^{\frac{1}{2}}t}dt$, we can apply the eigendecomposition $\Sigma = UDU^T$
$$
\int_0^{\infty} e^{-(UDU^T)^{\frac{1}{2}}t}S e^{-(UDU^T)^{\frac{1}{2}}t}dt = U \Big(\int_0^{\infty} e^{-D^{\frac{1}{2}}t} U^T S U e^{-D^{\frac{1}{2}}t} dt \Big)U^T 
= U(T\circ U^T S U)U^T
$$
where $\circ$ is Hadamard product and $T_{ij} = \frac{1}{\sqrt{\lambda_i}+\sqrt{\lambda_j}}$.

Thus, using   
$\Sigma_y^{\frac{1}{2}}\Sigma_k\Sigma_y^{\frac{1}{2}} = U_kD_kU_k^T$,  $\Sigma_y = U_yD_yU_y^T$  
and the corresponding $T_k,T_y$, we obtain
$$(\nabla_{\Sigma_y}F)(S) = \sum\limits_{k=1}^K P_k U_k\Big[T_k\circ U_k^T \Big[ U_y\Big(T_y\circ U_y^T S U_y\Big)U_y^T \Sigma_k\Sigma_y^{\frac{1}{2}} + \Sigma_y^{\frac{1}{2}}\Sigma_k U_y\Big(T_y\circ U_y^T S U_y\Big)U_y^T\Big] U_k \Big] U_k^T - S$$

To check non-singularity, we set $(\nabla_{\Sigma_y}F)(S) = 0$ and vectorize the equation to disentangle $S$. We apply the identity that $vec(AXB) = (B^T\otimes A)vec(X)$ where $\otimes$ is the Kronecker product \cite{horn2012matrix}. Meanwhile, vectorizing the Hadamard product yields
$$vec(T\circ X)  = diag(vec(T))vec(X) = (D^{\frac{1}{2}}\otimes I + I \otimes D^{\frac{1}{2}})^{-1} vec(X).$$
Then, the equation $(\nabla_{\Sigma_y}F)(S) = 0$ becomes,
\begin{align*}
vec(S) &= vec\Bigg\{ \sum\limits_{k=1}^K P_k U_k\Big[T_k\circ U_k^T \Big[ U_y\Big(T_y\circ U_y^T S U_y\Big)U_y^T \Sigma_k\Sigma_y^{\frac{1}{2}} + \Sigma_y^{\frac{1}{2}}\Sigma_k U_y\Big(T_y\circ U_y^T S U_y\Big)U_y^T\Big] U_k \Big] U_k^T \Bigg\}\\
&= \sum\limits_{k=1}^K P_k (U_k \otimes U_k) (D_k^{\frac{1}{2}}\otimes I + I \otimes D_k^{\frac{1}{2}})^{-1} (U_k^T \otimes U_k^T) \big( \Sigma_y^{\frac{1}{2}}\Sigma_k\otimes I + I \otimes \Sigma_k\Sigma_y^{\frac{1}{2}}\big)\\
&\qquad (U_y \otimes U_y) (D_y^{\frac{1}{2}}\otimes I + I \otimes D_y^{\frac{1}{2}})^{-1} (U_y^T \otimes U_y^T) vec(S)
\end{align*}
Splitting the term $\Sigma_y^{\frac{1}{2}}\Sigma_k\otimes I = (\Sigma_y^{\frac{1}{2}}\Sigma_k\Sigma_y^{\frac{1}{2}}\otimes I)\cdot (\Sigma_y^{-\frac{1}{2}}\otimes I)$, we get
\begin{align*}
vec(S) &= \bigg\{\Big[\sum\limits_{k=1}^K P_k (U_k \otimes U_k) (D_k^{\frac{1}{2}}\otimes I + I \otimes D_k^{\frac{1}{2}})^{-1} \big(D_k \otimes I)  (U_k^T \otimes U_k^T)\Big] \cdot\\
&\quad\Big[(U_y \otimes U_y) (D_y^{-\frac{1}{2}} \otimes I) (D_y^{\frac{1}{2}}\otimes I + I \otimes D_y^{\frac{1}{2}})^{-1} (U_y^T \otimes U_y^T)\Big] + \\
&\quad\Big[\sum\limits_{k=1}^K P_k (U_k \otimes U_k) (D_k^{\frac{1}{2}}\otimes I + I \otimes D_k^{\frac{1}{2}})^{-1} \big(I \otimes D_k)  (U_k^T \otimes U_k^T)\Big]\cdot\\
&\quad\Big[(U_y \otimes U_y) (I \otimes D_y^{-\frac{1}{2}}) (D_y^{\frac{1}{2}}\otimes I + I \otimes D_y^{\frac{1}{2}})^{-1} (U_y^T \otimes U_y^T)\Big]\bigg\} vec(S).
\end{align*}
Applying the three identities
$$D_k\otimes I = (D_k^{\frac{1}{2}}\otimes I + I\otimes D_k^{\frac{1}{2}})(D_k^{\frac{1}{2}}\otimes I - I\otimes D_k^{\frac{1}{2}}) + I\otimes D_k,$$
$$I\otimes D_k = (D_k^{\frac{1}{2}}\otimes I + I\otimes D_k^{\frac{1}{2}})^2 - D_k^{\frac{1}{2}}\otimes I - 2 D_k^{\frac{1}{2}}\otimes D_k^{\frac{1}{2}},$$
$$D_y^{-\frac{1}{2}}\otimes D_y^{-\frac{1}{2}} = (D_y^{\frac{1}{2}}\otimes I + I \otimes D_y^{\frac{1}{2}})^{-1}(D_y^{-\frac{1}{2}}\otimes I + I \otimes D_y^{-\frac{1}{2}}),$$
we obtain
\begin{align*}
vec(S) &= \bigg\{\Big[\sum\limits_{k=1}^K P_k (U_k \otimes U_k) (D_k^{\frac{1}{2}}\otimes I - I \otimes D_k^{\frac{1}{2}}) (U_k^T \otimes U_k^T)\Big] \cdot\\
&\quad\Big[(U_y \otimes U_y) (D_y^{-\frac{1}{2}} \otimes I) (D_y^{\frac{1}{2}}\otimes I + I \otimes D_y^{\frac{1}{2}})^{-1} (U_y^T \otimes U_y^T)\Big]\\
&\quad + \Big[\sum\limits_{k=1}^K P_k (U_k \otimes U_k) (D_k^{\frac{1}{2}}\otimes I + I \otimes D_k^{\frac{1}{2}})^{-1} \big(I \otimes D_k)  (U_k^T \otimes U_k^T)\Big]\cdot\\
&\quad\Big[(U_y \otimes U_y) (D_y^{-\frac{1}{2}}\otimes D_y^{-\frac{1}{2}}) (U_y^T \otimes U_y^T)\Big]\bigg\} vec(S)\\
&= \bigg\{\Big[\sum\limits_{k=1}^K P_k (U_k \otimes U_k) (D_k^{\frac{1}{2}}\otimes I + I \otimes D_k^{\frac{1}{2}}) (U_k^T \otimes U_k^T)\Big] \cdot \Big[(U_y \otimes U_y) (D_y^{-\frac{1}{2}}\otimes D_y^{-\frac{1}{2}}) (U_y^T \otimes U_y^T)\Big]\\
&\quad - \Big[\sum\limits_{k=1}^K P_k (U_k \otimes U_k) (D_k^{\frac{1}{2}}\otimes I) (U_k^T \otimes U_k^T)\Big] \cdot \Big[(U_y \otimes U_y) (I \otimes D_y^{-\frac{1}{2}}) (D_y^{\frac{1}{2}}\otimes I + I \otimes D_y^{\frac{1}{2}})^{-1} (U_y^T \otimes U_y^T)\Big]\\
&\quad - \Big[\sum\limits_{k=1}^K P_k (U_k \otimes U_k) (I \otimes D_k^{\frac{1}{2}}) (U_k^T \otimes U_k^T)\Big] \cdot \Big[(U_y \otimes U_y) (D_y^{-\frac{1}{2}} \otimes I) (D_y^{\frac{1}{2}}\otimes I + I \otimes D_y^{\frac{1}{2}})^{-1} (U_y^T \otimes U_y^T)\Big]\\
&\quad - \Big[\sum\limits_{k=1}^K P_k (U_k \otimes U_k) (D_k^{\frac{1}{2}}\otimes I + I \otimes D_k^{\frac{1}{2}})^{-1}(D_k^{\frac{1}{2}}\otimes D_k^{\frac{1}{2}}) (U_k^T \otimes U_k^T)\Big] \cdot\\
&\qquad\Big[(U_y \otimes U_y) (D_y^{-\frac{1}{2}}\otimes D_y^{-\frac{1}{2}}) (U_y^T \otimes U_y^T)\Big]\bigg\} vec(S)\\
&= \bigg\{\big(\Sigma_y \otimes I\big) \cdot \Big[(U_y \otimes U_y) (D_y^{-\frac{1}{2}} \otimes I) (D_y^{\frac{1}{2}}\otimes I + I \otimes D_y^{\frac{1}{2}})^{-1} (U_y^T \otimes U_y^T)\Big]\\
&\quad + \big(I \otimes \Sigma_y\big) \cdot \Big[(U_y \otimes U_y) (I \otimes D_y^{-\frac{1}{2}}) (D_y^{\frac{1}{2}}\otimes I + I \otimes D_y^{\frac{1}{2}})^{-1} (U_y^T \otimes U_y^T)\Big]\\
&\quad - \Big[\sum\limits_{k=1}^K P_k (U_k \otimes U_k) (D_k^{\frac{1}{2}}\otimes I + I \otimes D_k^{\frac{1}{2}})^{-1}(D_k^{\frac{1}{2}}\otimes D_k^{\frac{1}{2}}) (U_k^T \otimes U_k^T)\Big] \cdot\\
&\quad\Big[(U_y \otimes U_y) (D_y^{-\frac{1}{2}}\otimes D_y^{-\frac{1}{2}}) (U_y^T \otimes U_y^T)\Big]\bigg\} vec(S)\\
&= \bigg\{ I - \Big[\sum\limits_{k=1}^K P_k (U_k \otimes U_k) (D_k^{\frac{1}{2}}\otimes I + I \otimes D_k^{\frac{1}{2}})^{-1}(D_k^{\frac{1}{2}}\otimes D_k^{\frac{1}{2}}) (U_k^T \otimes U_k^T)\Big] \cdot \big(\Sigma_y^{-\frac{1}{2}} \otimes \Sigma_y^{-\frac{1}{2}}\big) \bigg\} vec(S)
\end{align*}
Hence, it follows that
\begin{equation*}
\Big[\sum\limits_{k=1}^K P_k (U_k \otimes U_k) (D_k^{\frac{1}{2}}\otimes I + I \otimes D_k^{\frac{1}{2}})^{-1}(D_k^{\frac{1}{2}}\otimes D_k^{\frac{1}{2}}) (U_k^T \otimes U_k^T)\Big] \cdot \big(\Sigma_y^{-\frac{1}{2}} \otimes \Sigma_y^{-\frac{1}{2}}\big) vec(S) = O
\end{equation*}
Denote the lengthy matrix by $\big[\sum_{k=1}^K P_k Y_k \big] Y$. Then, $Y$ is positive-definite, and since
\begin{equation*}
(D_k^{\frac{1}{2}}\otimes I + I \otimes D_k^{\frac{1}{2}})^{-1}(D_k^{\frac{1}{2}}\otimes D_k^{\frac{1}{2}})
\end{equation*}
are diagonal with positive entries, $Y_k$ are also positive-definite, so the equation holds if and only if $S=O$. We conclude that the gradient $(\nabla_{\Sigma_y} F)$ is always non-singular, and the implicit function theorem implies that $\Sigma_y$ depends differentiably on $\{\Sigma_1,\dots \Sigma_K\} \subseteq \prod_{i=1}^K\mathcal{S}_d^+$.
\end{proof}

It follows that since $\Sigma_k$ depends differentiably on $P_k^i$, the derivatives $\partial \Sigma_y / \partial P_k^i$ exist.

\section{Computation of derivative}
\label{appendix: computation of derivative}
To solve for the gradient $\nabla_{P^i}Tr[\Sigma_y]$, set $\Lambda_k^i = \partial \Sigma_y / \partial P_k^i \in \mathcal{S}_d$ as an unknown variable. Rather artificially, define the term
\begin{align*}
\Omega_k^i &:= \frac{1}{P_k}\frac{\partial (P_k)^2\Sigma_k}{\partial P_k^i} = \frac{1}{P_k} \frac{\partial P_k \sum_{i=1}^N P^i_k (x^i-\overline{x}_k) \cdot (x^i-\overline{x}_k)^T }{\partial P_k^i}\\
&= \Sigma_k + (x_i-\overline{x_k})\cdot (x_i-\overline{x_k})^T + 2\sum_{i=1}^N P^i_k (x^i-\overline{x}_k) \cdot \Big(-\frac{\partial \overline{x}_k}{\partial P^i_k}\Big)^T\\
&= \Sigma_k + (x_i-\overline{x_k})\cdot (x_i-\overline{x_k})^T
\end{align*}
Taking partial derivative $\partial P_k^i$ on both sides of the covariance formula (\ref{barycenter covariance discrete}), we obtain
\begin{align*}
\Lambda_k^i &= \sum\limits_{h\neq k} \Big(\nabla\big(\Sigma_y^{\frac{1}{2}}\big((P_h)^2\Sigma_h\big)\Sigma_y^{\frac{1}{2}}\big)^{\frac{1}{2}}\Big)\Big(\big(\nabla\Sigma_y^{\frac{1}{2}}\big)(\Lambda_k^i)\big((P_h)^2\Sigma_h\big)\Sigma_y^{\frac{1}{2}} + \Sigma_y^{\frac{1}{2}}\big((P_h)^2\Sigma_h\big)\big(\nabla\Sigma_y^{\frac{1}{2}}\big)(\Lambda_k^i)\Big)\\
&\quad + \Big(\nabla\big(\Sigma_y^{\frac{1}{2}}\big((P_k)^2\Sigma_k\big)\Sigma_y^{\frac{1}{2}}\big)^{\frac{1}{2}}\Big)\Big(\big(\nabla\Sigma_y^{\frac{1}{2}}\big)(\Lambda_k^i)\big((P_k)^2\Sigma_k\big)\Sigma_y^{\frac{1}{2}} + \Sigma_y^{\frac{1}{2}}P_k\Omega_k^i\Sigma_y^{\frac{1}{2}} + \Sigma_y^{\frac{1}{2}}\big((P_k)^2\Sigma_k\big)\big(\nabla\Sigma_y^{\frac{1}{2}}\big)(\Lambda_k^i)\Big)\\
&= \sum\limits_{h=1}^K U_h\Big[(P_h)^{-1}T_h\circ U_h^T \Big[ U_y\Big(T_y\circ U_y^T \Lambda_k^i U_y\Big)U_y^T (P_h)^2\Sigma_h\Sigma_y^{\frac{1}{2}} + \Sigma_y^{\frac{1}{2}}(P_h)^2\Sigma_h U_y\Big(T_y\circ U_y^T \Lambda_k^i U_y\Big)U_y^T\Big] U_h \Big] U_h^T\\
&\quad + U_k\Big[(P_k)^{-1} T_k\circ U_k^T \big(\Sigma_y^{\frac{1}{2}}P_k\Omega_k^i\Sigma_y^{\frac{1}{2}} + \Sigma_y^{\frac{1}{2}}\big) U_k \Big] U_k^T\\
&= \sum\limits_{h=1}^K P_h U_h\Big[T_h\circ U_h^T \Big[ U_y\Big(T_y\circ U_y^T \Lambda_k^i U_y\Big)U_y^T \Sigma_h\Sigma_y^{\frac{1}{2}} + \Sigma_y^{\frac{1}{2}}\Sigma_h U_y\Big(T_y\circ U_y^T \Lambda_k^i U_y\Big)U_y^T\Big] U_h \Big] U_h^T\\
&\quad + U_k\Big[T_k\circ U_k^T \big(\Sigma_y^{\frac{1}{2}}\Omega_k^i\Sigma_y^{\frac{1}{2}} + \Sigma_y^{\frac{1}{2}}\big) U_k \Big] U_k^T
\end{align*}
Vectorize and simplify it by the previous computations,
\begin{align*}
vec(\Lambda_k^i) &= \bigg\{ I - \Big[\sum\limits_{h=1}^K P_h (U_h \otimes U_h) (D_h^{\frac{1}{2}}\otimes I + I \otimes D_h^{\frac{1}{2}})^{-1}(D_h^{\frac{1}{2}}\otimes D_h^{\frac{1}{2}}) (U_h^T \otimes U_h^T)\Big] \cdot \big(\Sigma_y^{-\frac{1}{2}} \otimes \Sigma_y^{-\frac{1}{2}}\big) \bigg\} \\
&\quad \cdot vec(\Lambda_k^i) + (U_k \otimes U_k) (D_k^{\frac{1}{2}}\otimes I + I \otimes D_k^{\frac{1}{2}})^{-1} (U_k^T \otimes U_k^T)(\Sigma_y^{\frac{1}{2}} \otimes \Sigma_y^{\frac{1}{2}}) vec(\Omega_k^i)
\end{align*}
Therefore,
\begin{align*}
vec(\Lambda_k^i) &= \big(\Sigma_y^{\frac{1}{2}} \otimes \Sigma_y^{\frac{1}{2}}\big)\\
&\quad \Big[\sum\limits_{h=1}^K P_h (U_h \otimes U_h) (D_h^{\frac{1}{2}}\otimes I + I \otimes D_h^{\frac{1}{2}})^{-1}(D_h^{\frac{1}{2}}\otimes D_h^{\frac{1}{2}}) (U_h^T \otimes U_h^T)\Big]^{-1} \\
&\quad \Big[(U_k \otimes U_k) (D_k^{\frac{1}{2}}\otimes I + I \otimes D_k^{\frac{1}{2}})^{-1} (U_k^T \otimes U_k^T)\Big](\Sigma_y^{\frac{1}{2}} \otimes \Sigma_y^{\frac{1}{2}}) \cdot vec(\Omega_k^i)
\end{align*}
Denote the solution by $\Lambda_k^i = vec^{-1}(W_k vec(\Omega_k^i))$, we obtain an expression for the gradient of the objective function
\begin{equation*}
\nabla_{P^i} Tr[\Sigma_y] = \sum_{k=1}^K Tr[\Lambda_k^i]\e_k= \sum_{k=1}^K vec(I)^T \cdot W_k \cdot vec(\Omega_k^i)\e_k
\end{equation*}

\section{Proof of Theorem \ref{thm: scatter formulation of total transport cost}}
\label{appendix: scatter formulation}

We closely follow the proof of the variance decomposition theorem in \cite{yang2019BaryNet}. Let us assume that $\ZS = \{1,\dots K\}$ with clusters $\rho_k$ and weights $P_k$, and that all clusters $\rho_k \in \PS_{2,ac}(\R^d)$. It follows from Theorem 6.1 of \cite{agueh2011barycenters} that the barycenter $\mu$ is unique and $\mu\in\PS_{2,ac}(\R^d)$. The general formula (\ref{scatter formulation of total transport cost}) will follow from the standard approximation method in \cite{yang2019BaryNet}.

By Theorem 2.1 of \cite{cuesta1996bounds}, as $\rho_k,\mu$ are all Gaussian, there exists a unique transport map $T_k$ that transports $\rho_k$ to $\mu$ and has the form (\ref{optimal affine transport maps}). Similarly, there are unique optimal transport maps $T_{kh}$ from any $\rho_k$ to any $\rho_h$. We denote them by
\begin{equation*}
T_k(x) = A_kx+b_k, ~T_{kh} = A_{kh}x + b_{kh}
\end{equation*}
where the $A$'s are positive-definite matrices. Consider $T^{-1}_{kh} \circ T_h^{-1}$: It is an affine map that transports $\rho_k$ back to itself and has a positive-definite matrix $A_{kh}^{-1}A_h^{-1}A_k$. Since $\rho_k$ is a non-degenerate Gaussian, there is only one such affine map, namely the identity (One could examine each eigensubspace of $A_{kh}^{-1}A_h^{-1}A_k$ and consider how it operates there). It follows that $T_h^{-1}\circ T_k$ is exactly the optimal transport map $T_{kh}$.

Define the random variable $Y\sim \mu$ and the random variables $X_k = T_k^{-1}(Y) \sim \rho_k$. Then, by the optimality of $T_k$,
\begin{equation*}
\int W_2^2\big(\rho(\cdot|z),\mu\big) d\nu(z) = \sum_{k=1}^K P_k \E\big[\|Y-X_k\|^2\big]
\end{equation*}
Meanwhile, the identity $T_h^{-1}\circ T_k = T_{kh}$ implies that $X_h = T_{kh}(X_k)$ and that $W_2^2(\rho_k,\rho_h) = \E\big[\|X_k-X_h\|^2\big]$. Hence,
\begin{equation*}
\iint W_2^2\big(\rho(\cdot|z_1),\rho(\cdot|z_2)\big) d\nu(z_1)d\nu(z_2) = \sum_{k,h=1}^K P_k P_h \E\big[\|X_k-X_h\|^2\big]
\end{equation*}
Yet, Remark 3.9 of \cite{agueh2011barycenters} implies that $Y$ is exactly the average of $X_k$:
\begin{equation*}
Y = \sum_{k=1}^K P_k X_k
\end{equation*}
which leads to
\begin{equation*}
\sum_{k=1}^K P_k \E\big[\|Y-X_k\|^2\big] = \sum_{k,h=1}^K P_k P_h \E\big[\|X_k-X_h\|^2\big]
\end{equation*}
Hence, we have established equality (\ref{scatter formulation of total transport cost}).

\section{Derivation of gradients in Section \ref{sec: continuous clustering}}
\label{appendix: continuous gradient}

For each $z$, the partial derivative of the conditional standard deviation $\sigma(z)$ times the marginal latent distribution $\nu$ with respect to the conditional latent distribution $\nu_i = \nu(z|x_i)$ is given by
\begin{align*}
\frac{\partial\big(\sigma(z)\nu(z)\big)}{\partial\nu_i} &= \frac{\sum_{j=1}^N ||x_j - \overline{x}(z)||^2 \nu_j + ||x_i - \overline{x}(z)||^2 \cdot \sum_{j=1}^N \nu_j - 2\sum_{j=1}^N \nu_j (x_j-\overline{x})^T\cdot\frac{\partial \overline{x}(z)}{\partial \nu_i} }{ 2N \Big(\sum_{j=1}^N \nu_j \cdot \sum_{j=1}^N ||x_j - \overline{x}(z)||^2 \nu_j\Big)^{\frac{1}{2}}}  \nonumber\\
&= \frac{1}{2N}\Big[\sigma(z)+\frac{||x_i-\overline{x}(z)||^2}{\sigma(z)}\Big]
\end{align*}

Meanwhile, for each $\nu_i$, the partial derivatives of (\ref{soft assignment over real line}) are given by
\begin{equation*}
\frac{\partial \nu_i(z)}{\partial \overline{z}_j} = \Big[ -\frac{\overline{z}_j}{||\overline{z}||^3} + \frac{N}{\alpha^2}\frac{\overline{z}_j ||z-\overline{z}_i||^2}{||\overline{z}||^4} + \mathbf{1}_{i=j} \frac{N}{\alpha^2} \frac{z-\overline{z}_i}{||\overline{z}||^2} \Big] \nu_i(z)
\end{equation*}
So the Jacobian matrix $J_{\overline{z}}p$ of the vector $p:=[\nu_i]$ with respect to the parameter vector $\overline{z}:=[\overline{z}_j]$ is
\begin{align*}
J^T_{\overline{z}}p(z) &= \overline{z} \cdot \Big[ \big( -\frac{1}{||\overline{z}||^3} + \frac{N}{\alpha^2}\frac{||z-\overline{z}_i||^2}{||\overline{z}||^4} \big) \nu_i(z) \Big]^T_i
+ \text{diag}\Big( \frac{N}{\alpha^2} \frac{z-\overline{z}_i}{||\overline{z}||^2} \nu_i(z) \Big)\\
&= \overline{z} \cdot \Big[ \big( -\frac{1}{||\overline{z}||} + \frac{||z-\overline{z}_i||^2}{\epsilon^2} \big) \frac{\nu_i(z)}{||\overline{z}||^2} \Big]^T_i
+ \text{diag}\Big(  \frac{z-\overline{z}_i}{\epsilon^2} \nu_i(z) \Big)
\end{align*}

To show that we can differentiate under the integral sign of $\sigma$ in (\ref{std variance form continuous}), we rewrite the standard deviation $\sigma(z)$ into
$$\sigma(z) = \Big[\frac{1}{2}\iint ||x-y||^2 d\rho(x|z)d\rho(y|z)\Big]^{\frac{1}{2}} = \Big(\frac{1}{2}\sum_{i,j=1}^N ||x_i-x_j||^2 \nu_i \nu_j\Big)^{\frac{1}{2}}
= \frac{\Big(\sum_{i,j=1}^N ||x_i-x_j||^2 \nu_i \nu_j \Big)^{\frac{1}{2}} }{ \sqrt{2} \sum_{i=1}^N \nu_i }$$
Define the matrix $D = [D_{ij}] = [||x_i-x_j||^2]$. Then, the distribution term $d\nu(z)$ in $\sigma$ can be cancelled, and we obtain a simpler formula for $\sigma$:
\begin{align*}
\sigma &= \frac{1}{\sqrt{2}N}\int \Big(\sum_{i,j=1}^N D_{ij} \nu_i\nu_j \Big)^{\frac{1}{2}} dz = \frac{1}{\sqrt{2}N} \int ||\sqrt{D}p|| dz
\end{align*}
It is straightforward to show that for each $\overline{z}_0 \in \mathbb{R}^N$ ($\overline{z} \neq \mathbf{0}$), there exists some compact neighborhood $\overline{U}$ ($\overline{z}_0 \in U^o \subseteq \overline{U} \subseteq \mathbb{R}^N-\{\mathbf{0}\}$) such that the integrand $||\sqrt{D}p||$ and its $\overline{z}$ gradient are uniformly bounded by some integrable function. Thereby, Theorem 3.2 of \cite{lang2013interchange} shows that at $\overline{z}_0$, we are allowed to take derivatives under the integral sign:
\begin{align*}
\nabla_{\overline{z}} \sigma
=& \int \nabla_{\overline{z}}(\sigma(z) \nu(z)) dz = \int \frac{\nabla_{\overline{z}}\big(\sigma(z)\nu(z)\big)}{\nu(z)} d\nu(z)\\
=& \int \frac{1}{\nu(z)} J^T_{\overline{z}}p(z) \cdot \nabla_{p}\big(\sigma(z)\nu(z)\big) d\nu(z)\\
\frac{\partial \sigma}{\partial \overline{z}_i}
=& \int \frac{1}{2N\cdot \nu(z)} \Big\{ \frac{\overline{z}_i}{||\overline{z}||^2} \sum_{j=1}^N \Big[\sigma(z)+\frac{||x_j-\overline{x}(z)||^2}{\sigma(z)}\Big] \cdot \Big[ \frac{||z-\overline{z}_j||^2}{\epsilon^2} - \frac{1}{||\overline{z}||} \Big] \nu_j(z)\\
&+ \frac{z-\overline{z}_i}{\epsilon^2} \nu_i(z) \Big[\sigma(z)+\frac{||x_i-\overline{x}(z)||^2}{\sigma(z)}\Big] \Big\} d\nu(z)
\end{align*}
In particular, since the Jacobian $J_{\overline{z}}p$ consists of a rank-one matrix and a diagonal matrix, computing the above integrand for any $z$ takes only linear time, $O(N)$.

\end{document}